\documentclass{article}

\usepackage{mathptmx}

\usepackage{geometry}
\usepackage{amssymb}
\usepackage{latexsym, amsmath, amscd,amsthm}
\usepackage{graphicx}
\usepackage[percent]{overpic}
\usepackage{units}
\usepackage{xcolor}
\definecolor{linkblue}{HTML}{003d73}
\definecolor{linkgreen}{HTML}{006161}
\definecolor{linkred}{HTML}{a11950}
\usepackage[pagebackref]{hyperref}
\hypersetup{
	pdftitle={New Superbridge Index Calculations from Non-Minimal Realizations},
	pdfauthor={Clayton Shonkwiler},
	pdfsubject={knot theory},
	pdfkeywords={knots, random polygons, superbridge index, knot invariants},
	colorlinks=true,
	linkcolor=linkblue,
	citecolor=linkgreen,
	urlcolor=linkred
}
\PassOptionsToPackage{caption=false,labelformat=empty}{subfig}
\usepackage[lofdepth]{subfig}
\usepackage[export]{adjustbox}
\usepackage{algorithm}
\usepackage{algorithmicx}
\usepackage{algpseudocode}
\usepackage{booktabs}
\usepackage{authblk}

\usepackage{supertabular,multicol,ifthen,multirow}

\usepackage{array}
\newcolumntype{L}[1]{>{\raggedright\let\newline\\\arraybackslash\hspace{0pt}}m{#1}}
\newcolumntype{C}[1]{>{\centering\let\newline\\\arraybackslash\hspace{0pt}}m{#1}}
\newcolumntype{R}[1]{>{\raggedleft\let\newline\\\arraybackslash\hspace{0pt}}m{#1}}

\usepackage{cite}
\usepackage[nameinlink]{cleveref}

\makeatletter
\let\mcnewpage=\newpage
\newcommand{\TrickSupertabularIntoMulticols}{%
\renewcommand\newpage{%
    \if@firstcolumn%
        \hrule width\linewidth height0pt%
            \columnbreak%
        \else%
          \mcnewpage%
        \fi%
}%
}
\makeatother

\newtheorem{theorem}{Theorem}

\newtheorem{proposition}[theorem]{Proposition}
\newtheorem{corollary}[theorem]{Corollary}

\theoremstyle{definition}

\newtheorem{conjecture}[theorem]{Conjecture}

\newcommand{\R}{\mathbb{R}}

\newcommand{\stick}{\operatorname{stick}}

\newcommand{\bridge}{\operatorname{b}}
\newcommand{\superbridge}{\operatorname{sb}}

\title{New Superbridge Index Calculations from Non-Minimal Realizations}
\author{Clayton Shonkwiler}
\affil{Department of Mathematics, Colorado State University, Fort Collins, CO}
\date{}

\begin{document}

\maketitle

\begin{abstract}
	Previous work~\cite{shonkwiler_new_2020} used polygonal realizations of knots to reduce the problem of computing the superbridge number of a realization to a linear programming problem, leading to new sharp upper bounds on the superbridge index of a number of knots. The present work extends this technique to polygonal realizations with an odd number of edges and determines the exact superbridge index of many new knots, including the majority of the 9-crossing knots for which it was previously unknown and, for the first time, several 12-crossing knots. Interestingly, at least half of these superbridge-minimizing polygonal realizations do not minimize the stick number of the knot; these seem to be the first such examples. \Cref{sec:table} gives a complete summary of what is currently known about superbridge indices of prime knots through 10 crossings and \Cref{sec:exact values} gives all knots through 16 crossings for which the superbridge index is known.
\end{abstract}

\section{Introduction} 
\label{sec:introduction}

Given a tamely embedded closed curve $\gamma$ in $\R^3$, its \emph{superbridge number} $\superbridge(\gamma)$ is the maximum number of local maxima of any projection of $\gamma$ to a line. For a knot type $K$, its \emph{superbridge index} $\superbridge[K]$ is the minimum superbridge number of any realization of $K$. This knot invariant was first defined by Kuiper~\cite{kuiper_new_1987}, and was the first example of a so-called \emph{superinvariant}~\cite{adams_introduction_2002,Adams:2009ha}. 

Although Kuiper determined it for all torus knots, the superbridge index is generally quite hard to compute; for example, prior to this work it was known for only 49 of the 249 nontrivial knots in the Rolfsen table. The goal here is to determine the superbridge index of a number of knots, including 15 of the 27 9-crossing knots for which it was previously unknown and, for the first time, for some 12-crossing knots.

\begin{theorem}\label{thm:main}
	The knots $9_3$, $9_4$, $9_6$, $9_9$, $9_{11}$, $9_{13}$, $9_{17}$, $9_{18}$, $9_{22}$, $9_{23}$, $9_{25}$, $9_{27}$, $9_{30}$, $9_{31}$, and $9_{36}$ have superbridge index equal to 4, and the knots $11n_{72}$, $11n_{77}$, $12n_{60}$, $12n_{66}$, $12n_{219}$, $12n_{225}$, and $12n_{553}$ have superbridge index equal to 5.
\end{theorem}

\begin{figure}[t]
	\centering
		\subfloat[$9_{25}$]{\includegraphics[height=2in,width=2in,keepaspectratio,valign=c]{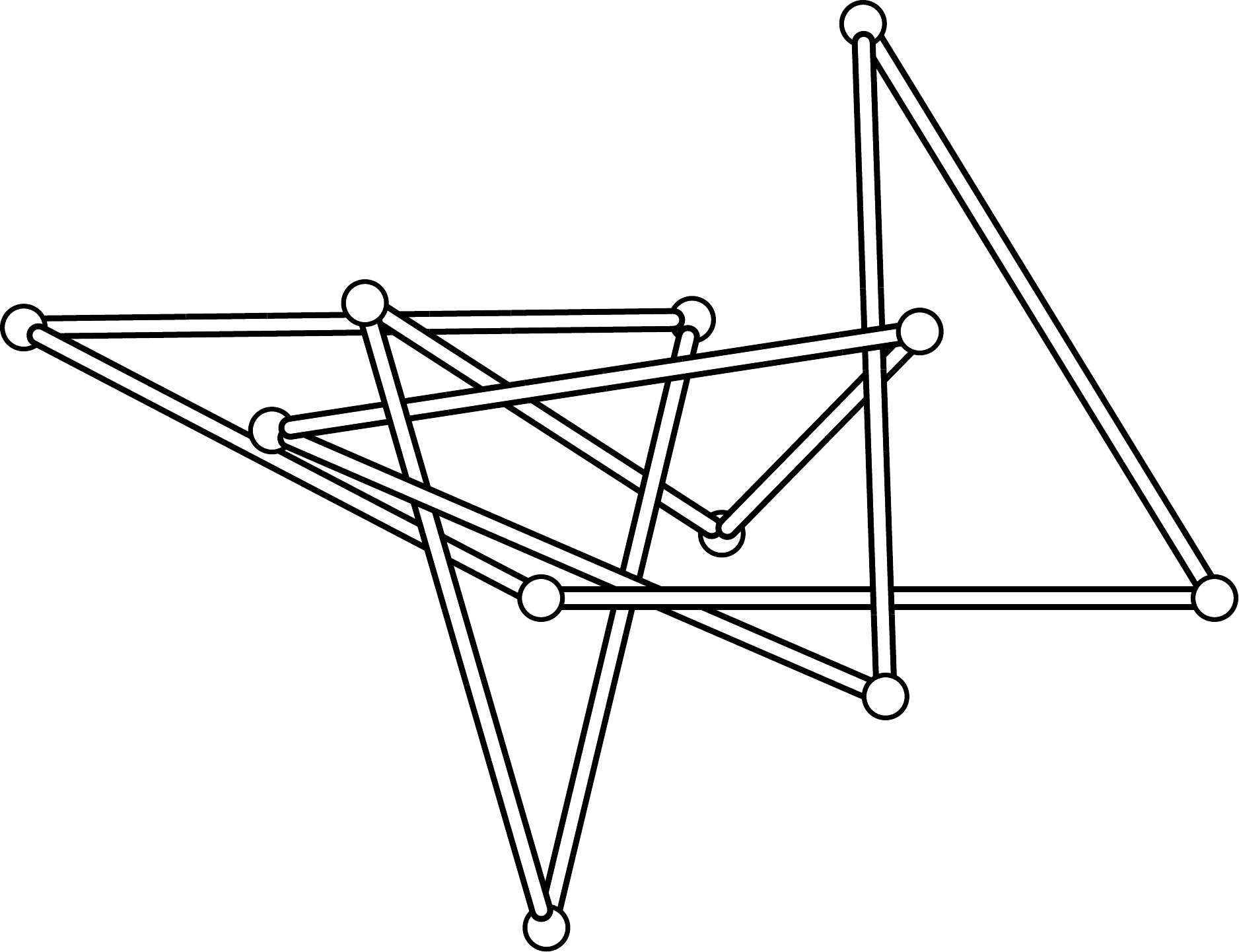}} \qquad \qquad
		\subfloat[$12n_{66}$]{\includegraphics[height=2in,width=2in,keepaspectratio,valign=c]{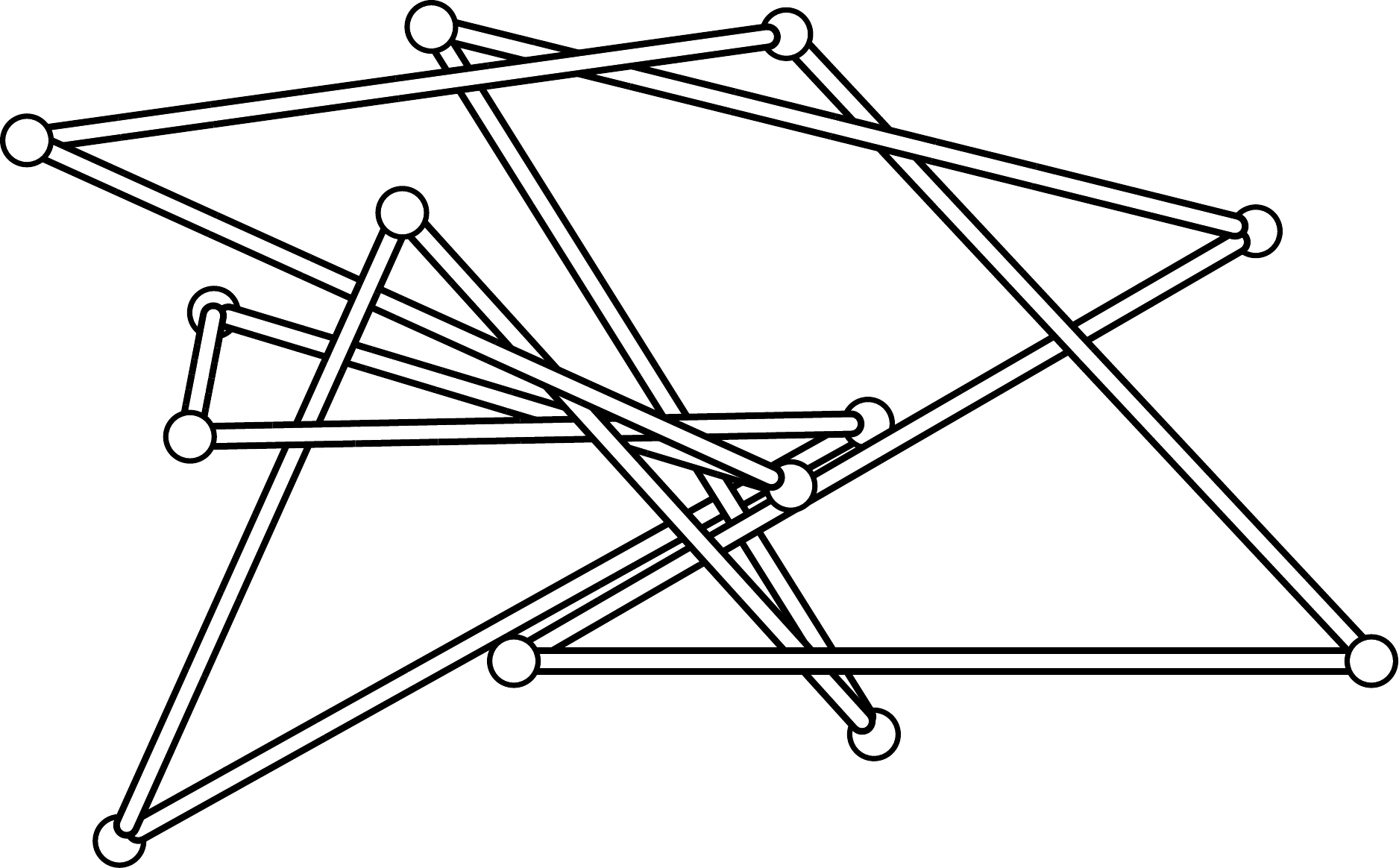}}
	\caption{Superbridge-minimizing realizations of $9_{25}$ and $12n_{66}$. Both are shown in orthographic perspective, viewed from the direction of the positive $z$-axis relative to the vertex coordinates given in \Cref{sec:coords}.}
	\label{fig:9_25 and K12n66}
\end{figure}

For each knot, the strategy is to find a (polygonal) realization of the knot with superbridge number equal to a known lower bound; the realizations of $9_{25}$ and $12n_{66}$ are shown in \Cref{fig:9_25 and K12n66}, and visualizations and coordinates for all knots mentioned in the theorem are given in \Cref{sec:coords}. These realizations were found by generating very large ensembles of polygonal knots in tight confinement using the approach described in~\cite{TomClay} and implemented in {\tt stick-knot-gen}~\cite{stick-knot-gen}. Overall, I generated more than 1 trillion random 9-, 10-, 11-, 12-, and 13-gons in the search for these examples. Interestingly, at least half of these superbridge-minimizing examples do not minimize the stick number of the knot. For example, the polygonal realization of $9_{25}$ shown in \Cref{fig:9_25 and K12n66} has 11 edges, but there are polygonal realizations of $9_{25}$ with only 10 edges~\cite{TomClay}.

Aside from the computational challenge of generating such large ensembles and determining knot types, the main difficulty is to identify potential examples and to rigorously compute their superbridge numbers. Different knots in the statement of the theorem require different arguments, so the proof will proceed in the following three sections, each of which groups together knots requiring the same argument, sorted in order of increasing complexity. Specifically, \Cref{thm:main} is a direct consequence of \Cref{cor:11n72 and 12n553}, \Cref{prop:9_22 11n77 12n60 12n219}, and \Cref{prop:rest}.

For reference, \Cref{sec:table} gives complete, up-to-date information about the superbridge index of all knots from the Rolfsen table, and \Cref{sec:exact values} lists all prime knots through 16 crossings for which the exact superbridge index is known. This information, along with the coordinate data from \Cref{sec:coords}, is also available from the {\tt stick-knot-gen} project~\cite{stick-knot-gen}.


\section{Stick number bounds} 
\label{sec:stick bounds}

The easiest way to get an upper bound on superbridge index is using Jin's bound relating superbridge index to stick number. Recall that the \emph{stick number} $\stick[K]$ of a knot $K$ is the minimum number of edges needed for any polygonal realization of the knot. 

\begin{theorem}[Jin~\cite{jin_polygon_1997}]\label{thm:JinBound}
	For any knot $K$, $\superbridge[K] \leq \frac{1}{2} \stick[K]$.
\end{theorem}

Both $11n_{72}$ and $12n_{553}$ have stick number no bigger than 11.

\begin{proposition}\label{prop:11n72 and 12n553}
	$\stick[11n_{72}],\stick[12n_{553}] \leq 11$.
\end{proposition}

\begin{proof}
	11-stick realizations of both knots are shown in \Cref{fig:11n72 and 12n553}. The coordinates of these realizations are given in \Cref{sec:coords}.
\end{proof}

\begin{figure}[t]
	\centering
		\subfloat[$11n_{72}$]{\includegraphics[height=2in,width=2in,keepaspectratio,valign=c]{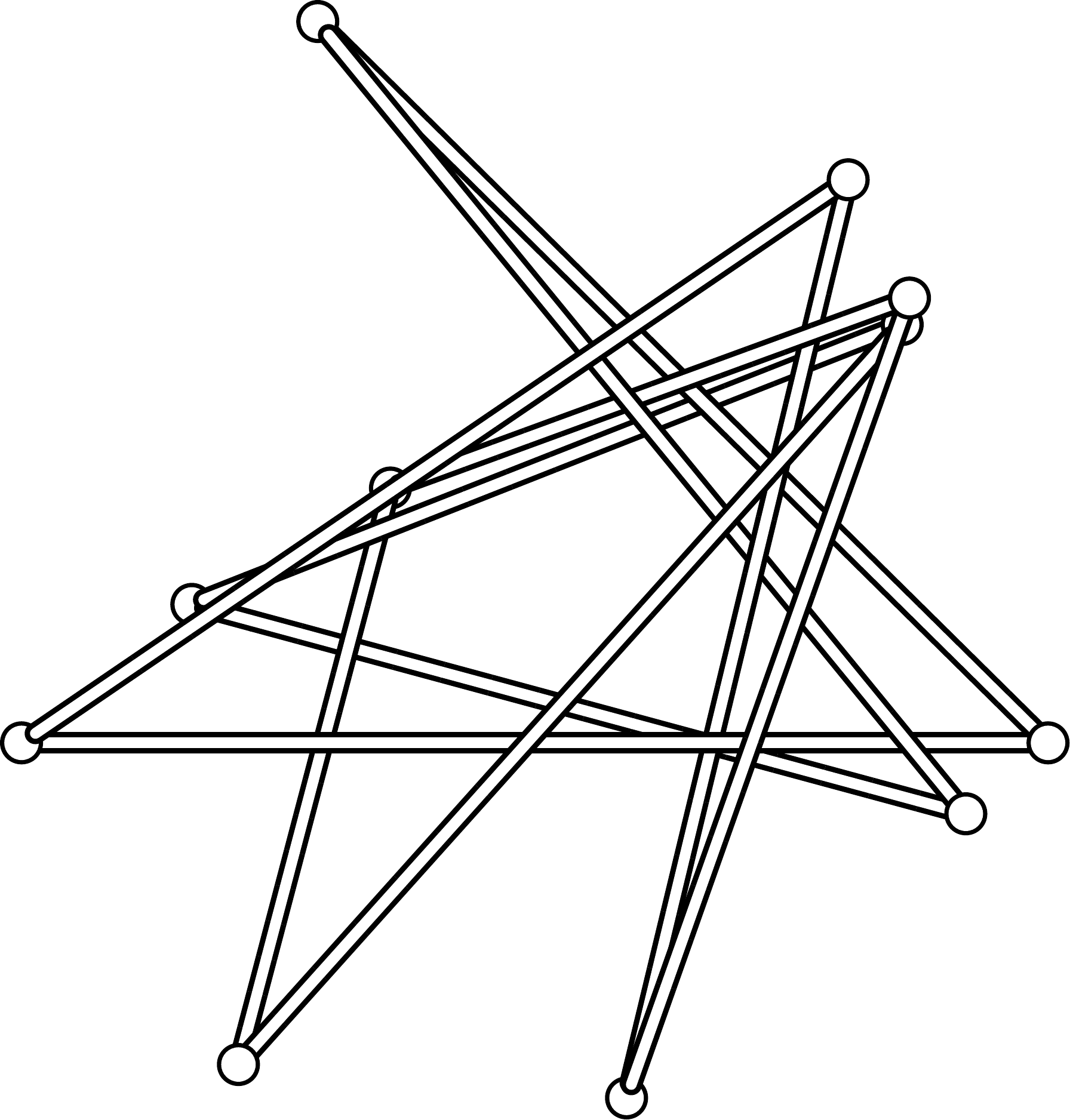}} \qquad \qquad
		\subfloat[$12n_{553}$]{\includegraphics[height=2in,width=2in,keepaspectratio,valign=c]{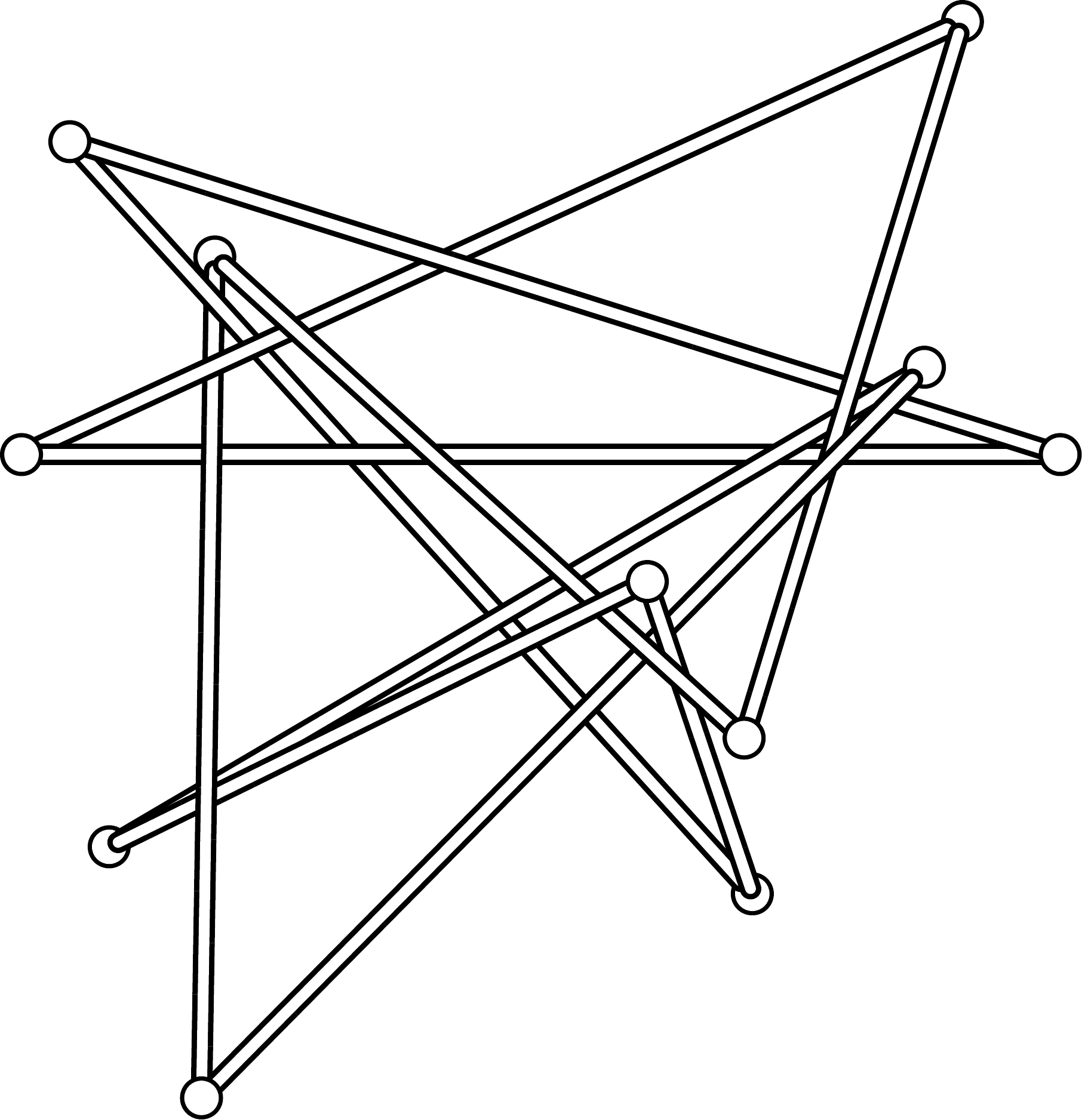}}
	\caption{11-stick realizations of $11n_{72}$ and $12n_{553}$. Both are shown in orthographic perspective, viewed from the direction of the positive $z$-axis relative to the vertex coordinates given in \Cref{sec:coords}.}
	\label{fig:11n72 and 12n553}
\end{figure}

It is then an immediate consequence of \Cref{thm:JinBound} that these knots both have superbridge index $\leq 5$. On the other hand, a useful lower bound on superbridge index comes from the bridge index $\bridge[K]$.

\begin{theorem}[Kuiper~\cite{kuiper_new_1987}]\label{thm:bridge bound}
	For any nontrivial knot $K$, $\bridge[K] < \superbridge[K]$.
\end{theorem}

Since $11n_{72}$ and $12n_{553}$ are both 4-bridge knots~\cite{Musick:2012uz,blair_wirtinger_2020,knotinfo}, this completely determines the superbridge index of both knots.

\begin{corollary}\label{cor:11n72 and 12n553}
	$\superbridge[11n_{72}] = \superbridge[12n_{553}] = 5$.
\end{corollary}


\section{Linear programming bounds} 
\label{sec:linear programming bounds}

The proof of Jin's bound $\superbridge[K] \leq \frac{1}{2}\stick[K]$ (\Cref{thm:JinBound}) is straightforward: the projection of any polygonal closed curve to a line cannot have more critical points than there are vertices along the original curve. In other words, the superbridge number of any $n$-edge polygonal realization of a knot is no more than $\frac{n}{2}$. This is more significant than it might first appear, since it is generally quite challenging to give any useful upper bound on the superbridge number of a closed curve. After all, if the projection of a closed curve $\gamma$ in to a line has $k$ local maxima, this shows that $\superbridge(\gamma) \geq k$: the inequality goes the wrong way! So it is a special situation to have an easily computed certificate (in this case, the count of edges) that a particular realization of a knot has superbridge number $\leq m$ for some concrete $m$.

The problem is that Jin's bound can be arbitrarily bad: while there are only finitely many knots $K$ with $\stick[K] \leq n$ for any $n$~\cite{calvo_geometric_2001,negami_ramsey_1991}, there are infinitely many knots with superbridge index equal to 4 (including all $(2,p)$-torus knots~\cite{kuiper_new_1987}). So it is frequently necessary to find some other way of certifying an upper bound on the superbridge number of a realization.

In a previous paper~\cite{shonkwiler_new_2020}, I developed a new approach based on linear programming, as follows. Suppose a polygonal closed curve has edge vectors $\vec{e}_1, \dots, \vec{e}_n$. Then for any $\vec{v} \in \R^3$, the number of local maxima of the projection of the curve to the line spanned by $\vec{v}$ is the number of times that $\vec{v} \cdot \vec{e}_i$ changes from positive to negative (computed cyclically, so that $\vec{v} \cdot \vec{e}_n > 0$ and $\vec{v} \cdot \vec{e}_1 < 0$ contributes 1 to the count). Jin's bound implies that $\superbridge(\vec{e}_1, \dots , \vec{e}_n) \leq \frac{n}{2}$, with equality if and only if $n$ is even and there is some line onto which every vertex projects to a local minimum or a local maximum; i.e., $\vec{v} \cdot \vec{e}_1, \dots , \vec{v} \cdot \vec{e}_n$ alternates signs. Since $\vec{v}$ can be replaced by $-\vec{v}$, it is no restriction to assume that the alternating sign pattern is $+, -, \dots , +, -$. 

Equivalently, if
\begin{equation}\label{eq:E}
	E := \left[\, \vec{e}_1 \,\,\vline\, -\vec{e}_2 \,\,\vline\,\, \cdots \,\,\vline\,\, \vec{e}_{2k-1} \,\,\vline\, -\vec{e}_{2k} \,\right],
\end{equation}
then $\superbridge(\vec{e}_1 , \dots , \vec{e}_{2k}) < k$ if and only if there is no $\vec{v} \in \R^3$ so that $\vec{v}^T E$ has all positive entries. In other words, infeasibility of a system of linear inequalities provides a stronger upper bound on the superbridge number of a polygonal realization of a knot than that coming from Jin's bound.

This is helpful, because classical theorems of the alternative~\cite{Tucker:1956ve,dantzig_linear_2003} control feasibility of such systems. For example:

\begin{theorem}[Gordan~\cite{gordan_uber_1873}]\label{thm:Gordan}
	If $A$ is a $k \times \ell$ matrix, then exactly one of the following is true:
	\begin{enumerate}
		\item There is $\vec{v} \in \R^k$ so that $\vec{v}^T A$ has all positive entries.
		\item There is a nonzero vector $\vec{u} \in \R^\ell$ with nonnegative entries so that $A \vec{u} = \vec{0}$.
	\end{enumerate}
\end{theorem}

This yields the desired bound on superbridge numbers.

\begin{corollary}[Shonkwiler~\cite{shonkwiler_new_2020}]\label{cor:Gordan}
	Suppose $\vec{e}_1, \dots , \vec{e}_{2k}$ are the edges of a closed polygonal curve in $\R^3$. Then $\superbridge(\vec{e}_1, \dots , \vec{e}_{2k}) < k$ if and only if there exists a nonzero vector $\vec{u} \in \R^n$ with nonnegative entries so that
	\begin{equation} \label{eq:even gordan}
		E \vec{u} = \vec{0},
	\end{equation}
	with the matrix $E$ defined as in \eqref{eq:E}.
\end{corollary}

The vector $\vec{u}$ provides a certificate that the superbridge number of the polygonal realization is strictly less than half the number of vertices. This approach is enough to determine the superbridge indices of four more knots in the statement of \Cref{thm:main}.

\begin{proposition}\label{prop:9_22 11n77 12n60 12n219}
	$\superbridge[9_{22}] = 4$ and $\superbridge[11n_{77}] = \superbridge[12n_{60}] = \superbridge[12n_{219}] = 5$.
\end{proposition}

\begin{figure}[t]
	\centering
		\begin{tabular*}{0.75\textwidth}{C{2.2in} R{.4in} R{.4in} R{.4in}}
		\multicolumn{4}{c}{$9_{22}$} \\
		\cline{1-4}\noalign{\smallskip}
		\multirow{12}{*}{\includegraphics[height=1.8in,width=1.8in,keepaspectratio]{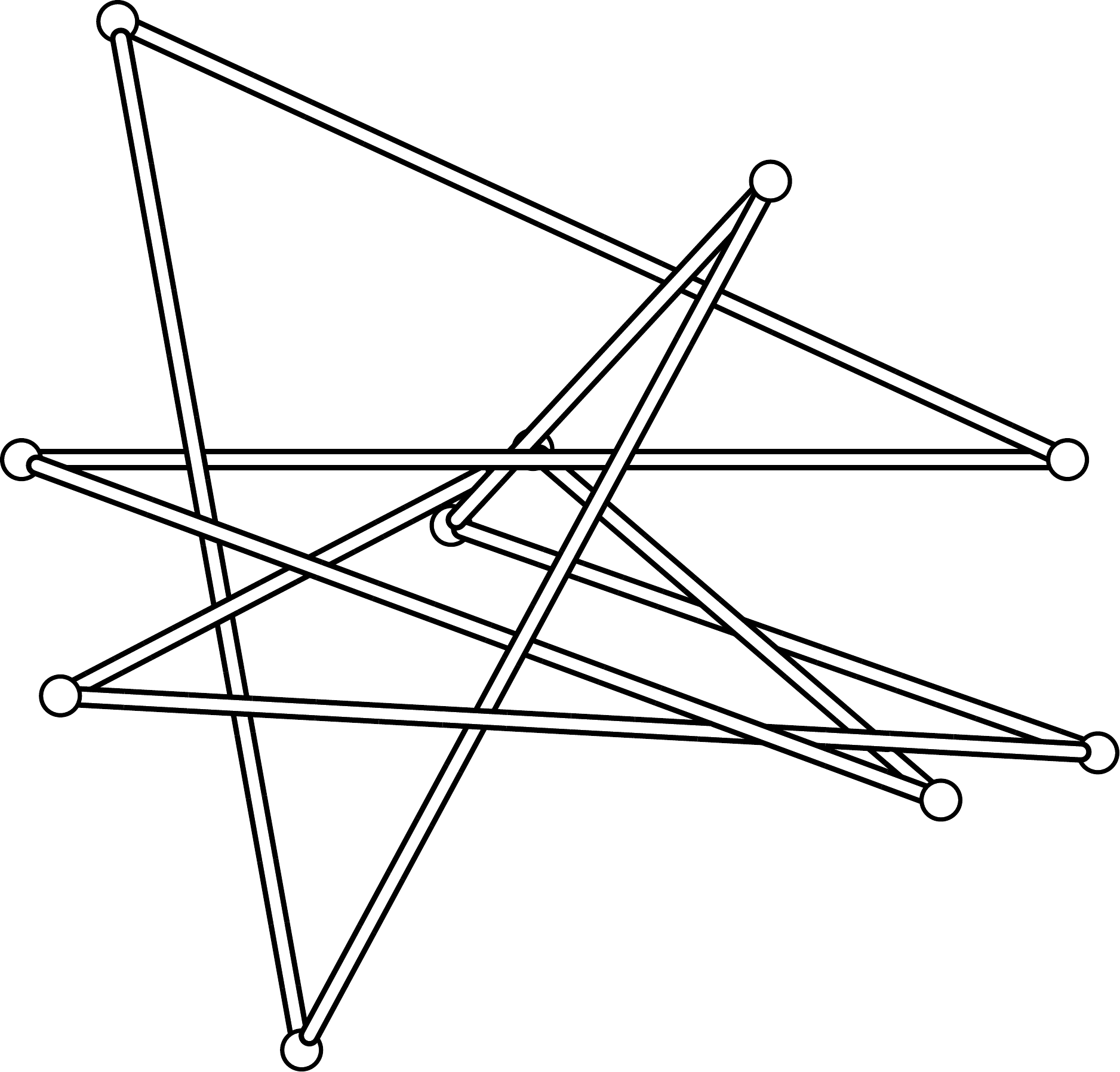}}
		& $0$ & $0$ & $0$ \\
		& $1000$ & $0$ & $0$ \\
		& $92$ & $419$ & $0$ \\
		& $268$ & $-564$ & $44$ \\
		& $716$ & $266$ & $374$ \\
		& $411$ & $-63$ & $-519$ \\
		& $1029$ & $-280$ & $236$ \\
		& $37$ & $-226$ & $352$ \\
		& $489$ & $10$ & $-509$ \\
		& $879$ & $-326$ & $349$\\
		\end{tabular*}
		\medskip
		\[(107029574,1,1,1,2,97177084,37335363,57495926,18717108,9955666)\]
	\caption{A 10-stick $9_{22}$ with superbridge number $4$. The three columns to the right of the image give the coordinates of the vertices, and the vector below is the $\vec{u}$ solving~\eqref{eq:even gordan}. The knot is shown in orthographic perspective, viewed from the direction of the positive $z$-axis.}
	\label{fig:9_22}
\end{figure}

\begin{proof}
	By \Cref{thm:JeonJin}, stated below, $\superbridge[9_{22}] \geq 4$. On the other hand, \Cref{fig:9_22} gives a 10-stick realization of $9_{22}$ which has superbridge number $\leq 4$. To verify this with \Cref{cor:Gordan}, it suffices to find a nonzero vector $\vec{u} \in \R^{10}$ with nonnegative entries so that
	\[
		\begin{bmatrix}1000 & 908 & 176 & -448 & -305 & -618 & -992 & -452 & 390 & 879 \\
 0 & -419 & -983 & -830 & -329 & 217 & 54 & -236 & -336 & -326 \\
 0 & 0 & 44 & -330 & -893 & -755 & 116 & 861 & 858 & 349 \end{bmatrix} \vec{u} = \vec{0}.
	\] 
	It is easy to check that
	\[
		\vec{u} = (107029574, 1, 1, 1, 2, 97177084, 37335363, 57495926, 18717108, 9955666)
	\]
	is such a vector, so this completes the proof that $\superbridge[9_{22}] = 4$.
	
	A similar argument shows that the knots $11n_{77}$, $12n_{60}$, and $12n_{219}$ have superbridge index $\leq 5$: 12-stick realizations of these knots together with vectors $\vec{u}$ solving~\eqref{eq:even gordan} are given in \Cref{sec:coords}. The certificate vectors $\vec{u}$ were found using \emph{Mathematica}'s {\tt FindInstance} function. 
	
	Since $11n_{77}$, $12n_{60}$, and $12n_{219}$ are all 4-bridge knots~\cite{Musick:2012uz,blair_wirtinger_2020,knotinfo}, \Cref{thm:bridge bound} implies that each of these knots has superbridge index $\geq 5$, completing the proof that their superbridge indices are exactly 5.
\end{proof}

Each of the polygonal realizations used in the above proof was originally generated with coordinates given as double-precision floating point numbers. However, these coordinates were all rounded to three significant digits and converted to integers (while verifying this did not change the knot type) to make it easier to verify the existence of exact solutions to~\eqref{eq:even gordan}.

The lower bound on $\superbridge[9_{22}]$ comes from Jeon and Jin's characterization of the possible 3-superbridge knots.

\begin{theorem}[Jeon--Jin~\cite{jeon_computation_2002}]\label{thm:JeonJin}
	Every knot except $3_1$ and $4_1$ and possibly $5_2$, $6_1$, $6_2$, $6_3$, $7_2$, $7_3$, $7_4$, $8_4$, and $8_9$ has superbridge index $\geq 4$.
\end{theorem}

This statement is slightly different than the one given in Jeon and Jin's paper, which included $8_7$ among the possible 3-superbridge knots; see~\cite{shonkwiler_new_2020} for the proof that $\superbridge[8_7]=4$.


\section{An extension to odd numbers of edges} 
\label{sec:extension}

Between \Cref{prop:9_22 11n77 12n60 12n219} and~\cite{shonkwiler_new_2020}, \Cref{cor:Gordan} has now been used to determine the superbridge indices of 21 8- and 9-crossing knots, but as stated it is limited to polygonal realizations of knots with an even number of edges. The goal now is to extend this approach to polygonal knots with an odd number of edges.

Suppose $\vec{e}_1, \dots , \vec{e_{2k+1}}$ are the edge vectors of some closed polygonal curve in $\R^3$. By \Cref{thm:JinBound}, $\superbridge(\vec{e}_1, \dots , \vec{e}_{2k+1}) \leq k$. If this is actually an equality, then there must be some vector $\vec{v} \in \R^3$ so that the projection of the polygon to the line containing $\vec{v}$ has exactly $k$ local maxima, meaning that the list $\vec{v} \cdot \vec{e}_1, \dots, \vec{v} \cdot \vec{e}_{2k+1}$ switches from positive to negative $k$ times (when considered cyclically). After possibly replacing $\vec{v}$ with $-\vec{v}$, the sign pattern can be assumed to be some cyclic permutation of 
\[
	+, +, -, +, -, \dots , +, -.
\]

Let $(s_1, \dots , s_{2k+1})=(0,0,1,0,1,\dots,0,1)$ and, for each $j \in \{1, \dots , 2k+1\}$, define the matrix
\[
	E_j := \left[\,(-1)^{s_{1+j}} \vec{e}_1 \,\,\vline\,\, (-1)^{s_{2+j}}\vec{e}_2 \,\,\vline\,\, \cdots \,\,\vline\, (-1)^{s_{2k+1+j}} \vec{e}_{2k+1}\,\right],
\]
where the subscripts are computed cyclically (i.e., $s_{2k+2} = s_1$, $s_{2k+3} = s_2$, etc.). 

Then $\superbridge(\vec{e}_1, \dots , \vec{e}_{2k+1}) = k$ if and only if there exist $\vec{v} \in \R^3$ and $j \in \{1, \dots , 2k+1\}$ so that $\vec{v}^T E_j$ has all positive entries. Contrapositively, none of the linear systems $\vec{v}^T E_j > 0$ have solutions if and only if $\superbridge(\vec{e}_1, \dots , \vec{e}_{2k+1}) < k$. Stated in this way, the following corollary of \Cref{thm:Gordan} gives new bounds on the superbridge indices of polygonal curves with an odd number of edges.

\begin{corollary}\label{cor:Gordan2}
	Suppose $\vec{e}_1, \dots , \vec{e}_{2k+1}$ are the edge vectors of a closed polygonal curve in $\R^3$. Then the curve has superbridge number $<k$ if and only if there exist nonzero vectors $\vec{u}_1 , \dots , \vec{u}_{2k+1} \in \R^{2k+1}$ with nonnegative entries solving the matrix equations
	\begin{equation}\label{eq:odd gordan}
		E_j \vec{u}_j = \vec{0}
	\end{equation}
	for all $j=1, \dots , 2k+1$. If $U = \left[\, \vec{u}_1 \,\vline\, \cdots \,\vline\, \vec{u}_{2k+1}\,\right]$, the $(2k+1)\times(2k+1)$ matrix $U$ serves as a computable certificate of the superbridge number bound. 
\end{corollary}

\Cref{cor:Gordan2} can now be used to determine the superbridge indices of the rest of the knots in \Cref{thm:main}.

	\begin{figure}[t]
		\centering
			\begin{tabular*}{0.75\textwidth}{C{2.2in} R{.4in} R{.4in} R{.4in}}
			\multicolumn{4}{c}{$9_{36}$} \\
			\cline{1-4}\noalign{\smallskip}
			\multirow{12}{*}{\includegraphics[height=1.8in,width=1.8in,keepaspectratio]{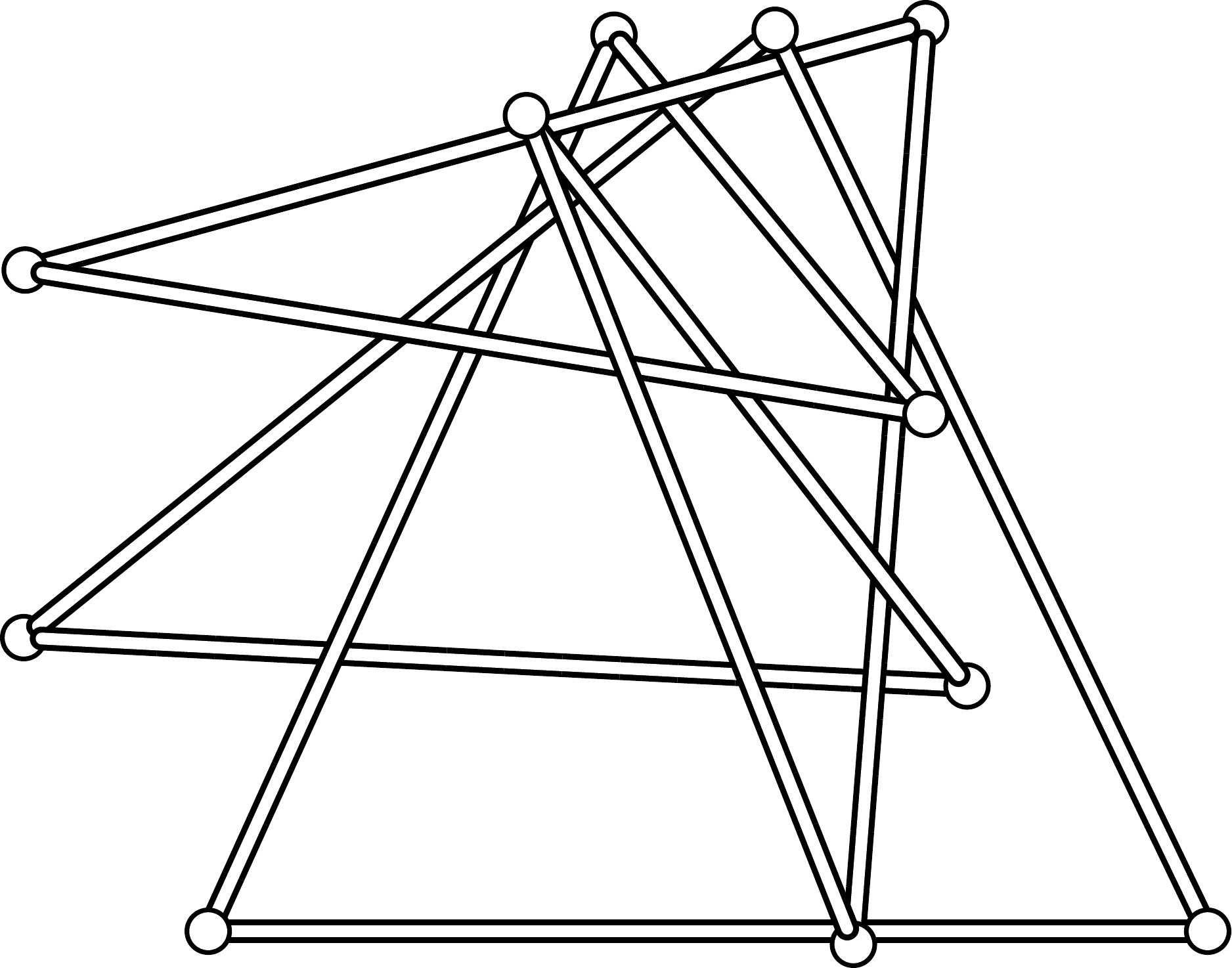}}
			& $0$ & $0$ & $0$ \\
			& $1000$ & $0$ & $0$ \\
			& $567$ & $902$ & $0$ \\
			& $-185$ & $294$ & $-256$ \\
			& $759$ & $245$ & $71$ \\
			& $318$ & $816$ & $764$ \\
			& $646$ & $-13$ & $311$ \\
			& $718$ & $908$ & $-72$ \\
			& $-183$ & $662$ & $285$ \\
			& $718$ & $517$ & $694$ \\
			& $406$ & $897$ & $-177$ \\
			\end{tabular*}
			 \[\begin{tiny}\begin{bmatrix}$1$ & $1$ & $1$ & $1$ & $1$ & $1$ & $1$ \
			& $1$ & $1$ & $1$ & $1$ \\
			$1$ & $1$ & $1$ & $1$ & $1$ & $1$ & $1$ & $12523027847$ & $1$ & $1$ & \
			$1$ \\
			$5426124359$ & $1$ & $1$ & $1$ & $1424202531$ & $1$ & $1$ & $1$ & $1$ \
			& $6584192388$ & $8992884808$ \\
			$1$ & $1$ & $1$ & $12223967$ & $1$ & $1$ & $263408915$ & \
			$29465794187$ & $9986903746$ & $1$ & $1$ \\
			$495044898$ & $5412091732$ & $100403683$ & $1$ & $3$ & $212199095$ & \
			$1$ & $1$ & $1$ & $1$ & $1$ \\
			$1$ & $1$ & $1$ & $3$ & $2159516409$ & $1$ & $65852229$ & $1$ & $1$ & \
			$6627781141$ & $1$ \\
			$1$ & $1$ & $38906659$ & $19412935$ & $1$ & $53049774$ & $39336205$ & \
			$1$ & $1$ & $1$ & $1$ \\
			$1125551611$ & $1735612295$ & $29180733$ & $858339$ & $1557767516$ & \
			$11511945$ & $88068805$ & $355250231$ & $6884251289$ & $36701288$ & \
			$398991934$ \\
			$1241803067$ & $3170128948$ & $69866228$ & $11625485$ & $267607237$ & \
			$79522441$ & $20150485$ & $2006088206$ & $2575693266$ & $1056878715$ \
			& $2505218830$ \\
			$45910998$ & $355246092$ & $26475285$ & $6280015$ & $89877512$ & \
			$374500950$ & $1677801$ & $145553542$ & $1448979271$ & $824513076$ & \
			$3204959462$ \\
			$3901711428$ & $1642348279$ & $12533499$ & $2078963$ & $40849202$ & \
			$426090556$ & $184771845$ & $27474391863$ & $3948141106$ & \
			$5211154861$ & $7748788486$\end{bmatrix}\end{tiny}\]
		\caption{Visualization of an 11-stick realization of $9_{36}$ with superbridge number $4$. The three columns to the right of the image give the coordinates of the vertices, and the matrix below is the $U$ whose columns are the~$\vec{u}_j$ satisfying~\eqref{eq:odd gordan}. The knot is shown in orthographic perspective, viewed from the direction of the positive $z$-axis.}
		\label{fig:9_36}
	\end{figure}

\begin{proposition}\label{prop:rest}
	The knots $9_3$, $9_4$, $9_6$, $9_9$, $9_{11}$, $9_{13}$, $9_{17}$, $9_{18}$, $9_{23}$, $9_{25}$, $9_{27}$, $9_{30}$, $9_{31}$, and $9_{36}$ have superbridge index equal to 4, and $\superbridge[12n_{66}] = \superbridge[12n_{225}] = 5$.
\end{proposition}

\begin{proof}
	Each of these knots has superbridge index $\geq 4$ by \Cref{thm:JeonJin}, so for the 9-crossing knots it suffices to find 11-stick realizations and certificate matrices $U$ showing they have superbridge number $\leq 4$, as in \Cref{cor:Gordan2}. $12n_{66}$ and $12n_{225}$ are both 4-bridge knots~\cite{blair_wirtinger_2020,knotinfo}, so their superbridge indices are $\geq 5$ and it suffices to find 13-stick realizations with corresponding $U$.
	
	These realizations and certificate matrices are given in \Cref{sec:coords}. The entry for $9_{36}$ is reproduced in \Cref{fig:9_36}. In each case, a visualization of the knot is shown next to the coordinates of the vertices, and the matrix $U$ is given below.
\end{proof}


\section{Conclusion} 
\label{sec:conclusion}

One virtue of \Cref{cor:Gordan2} is that it provides examples of superbridge-minimizing polygonal knots which do not minimize stick number. Specifically, all of the 9-crossing knots mentioned in \Cref{prop:rest} except $9_6$, $9_{23}$, and $9_{36}$ are known to have stick number $\leq 10$~\cite{TomClay,rawdon_upper_2002,shonkwiler_all_2021}, but to my knowledge there are no 10-stick realizations of these knots with superbridge number 4. For example, in the course of this project I generated 1147 random 10-stick realizations of $9_{27}$ (out of more than 170 billion random 10-gons), none of which seem to have superbridge number equal to 4. This suggests the very plausible but still intriguing possibility that there may exist knots for which the stick number and superbridge index cannot be achieved by the same realization.

\begin{conjecture}
	There exists a knot type $K$ for which no polygonal realization achieves both $\stick[K]$ and $\superbridge[K]$.
\end{conjecture} 

The algorithm used to generate large ensembles of random polygons for this project produces \emph{equilateral} polygons---that is, closed polygonal curves for which all edges are the same length. While I don't know of any evidence either way, it is conceivable that minimizing superbridge number is easier with heterogeneous edgelengths, so it might be worthwhile to perform similar investigations with non-equilateral random polygons.


\subsection*{Acknowledgments} 
\label{sub:acknowledgments}

Thanks, as ever, to Allison Moore and Chuck Livingston for maintaining KnotInfo~\cite{knotinfo} and to Jason Cantarella and Tom Eddy for technical support. This work was partially supported by grants from the Simons Foundation (\#354225 and \#709150).


\appendix

\section{Stick Number and Superbridge Index Bounds}\label{sec:table}

Bounds on the superbridge index for prime knots through 10 crossings are given below, including references for where these results were proved. If an exact value is not known, the possible values, as determined by known upper and lower bounds, are given in the form of an interval; e.g., the entry $[3,4]$ for $5_2$ means that $3 \leq \superbridge[5_2] \leq 4$.

The lower bounds on superbridge index number always comes from either Kuiper's result $\bridge[K]<\superbridge[K]$ for nontrivial knots~\cite{kuiper_new_1987} (\Cref{thm:bridge bound}) or Jeon and Jin's characterization of possible 3-superbridge knots~\cite{jeon_computation_2002} (\Cref{thm:JeonJin}), so these references are not explicitly given. 

When the upper bound comes from Jin's bound $\superbridge[K] \leq \frac{1}{2} \stick[K]$ (\Cref{thm:JinBound}), Jin's paper~\cite{jin_polygon_1997} is cited along with the current best upper bound on stick number for the knot.

\clearpage

\Crefname{theorem}{Thm.}{Thms.}

\setlength{\tabcolsep}{6pt}

\begin{center}
\begin{multicols*}{3}
\TrickSupertabularIntoMulticols

\tablefirsthead{
	\multicolumn{1}{l}{$K$} & \multicolumn{2}{l}{$\superbridge[K]$} \\
	\midrule
}
\tablehead{
	\multicolumn{1}{l}{$K$} & \multicolumn{2}{l}{$\superbridge[K]$} \\
	\midrule
}
\tablelasttail{\bottomrule}

\begin{supertabular*}{.26\textwidth}{lll} \label{tab:stick numbers}
$ \!\!0_{1} $ & 1 & \\
$ 3_{1} $ & 3 & \cite{kuiper_new_1987} \\
$ 4_{1} $ & 3 & \cite{jin_polygon_1997, randell_elementary_1994} \\
$ 5_{1} $ & 4 & \cite{kuiper_new_1987} \\
$ 5_{2} $ & $[3,4]$ & \cite{jin_polygon_1997} \\
$ 6_{1} $ & $[3,4]$ & \cite{negami_ramsey_1991, meissen_edge_1998, jin_polygon_1997} \\
$ 6_{2} $ & $[3,4]$ & \cite{negami_ramsey_1991, meissen_edge_1998, jin_polygon_1997} \\
$ 6_{3} $ & $[3,4]$ & \cite{negami_ramsey_1991, meissen_edge_1998, jin_polygon_1997} \\
\midrule
$ 7_{1} $ & 4 & \cite{kuiper_new_1987}\\
$ 7_{2} $ & $[3,4]$ & \cite{meissen_edge_1998, jin_polygon_1997}\\
$ 7_{3} $ & $[3,4]$ & \cite{meissen_edge_1998, jin_polygon_1997}\\
$ 7_{4} $ & $[3,4]$ & \cite{meissen_edge_1998, jin_polygon_1997}\\
$ 7_{5} $ & 4 & \cite{meissen_edge_1998, jin_polygon_1997}\\
$ 7_{6} $ & 4 & \cite{meissen_edge_1998, jin_polygon_1997}\\
$ 7_{7} $ & 4 & \cite{meissen_edge_1998, jin_polygon_1997}\\
\midrule
$ 8_{1} $ & 4 & \cite{shonkwiler_new_2020} \\
$ 8_{2} $ & 4 & \cite{shonkwiler_new_2020} \\
$ 8_{3} $ & 4 & \cite{shonkwiler_new_2020} \\
$ 8_{4} $ & $[3,4]$ & \cite{shonkwiler_new_2020} \\
$ 8_{5} $ & 4 & \cite{shonkwiler_new_2020} \\
$ 8_{6} $ & 4 & \cite{shonkwiler_new_2020} \\
$ 8_{7} $ & 4 & \cite{shonkwiler_new_2020} \\
$ 8_{8} $ & 4 & \cite{shonkwiler_new_2020} \\
$ 8_{9} $ & $[3,4]$ & \cite{shonkwiler_new_2020} \\
$ 8_{10} $ & 4 & \cite{shonkwiler_new_2020} \\
$ 8_{11} $ & 4 & \cite{shonkwiler_new_2020} \\
$ 8_{12} $ & 4 & \cite{shonkwiler_new_2020} \\
$ 8_{13} $ & 4 & \cite{shonkwiler_new_2020} \\
$ 8_{14} $ & 4 & \cite{shonkwiler_new_2020} \\
$ 8_{15} $ & 4 & \cite{shonkwiler_new_2020} \\
$ 8_{16} $ & 4 & \cite{rawdon_upper_2002, jin_polygon_1997} \\
$ 8_{17} $ & 4 & \cite{rawdon_upper_2002, jin_polygon_1997} \\
$ 8_{18} $ & 4 &  \cite{calvo_geometric_2001, rawdon_upper_2002, jin_polygon_1997} \\
$ 8_{19} $ & 4 & \cite{kuiper_new_1987} \\
$ 8_{20} $ & 4 & \cite{negami_ramsey_1991, meissen_edge_1998, jin_polygon_1997} \\
$ 8_{21} $ & 4 & \cite{meissen_edge_1998, jin_polygon_1997}\\
\midrule
$ 9_{1} $ & 4 & \cite{kuiper_new_1987} \\
$ 9_{2} $ & $[4,5]$ & \cite{TomClay, jin_polygon_1997} \\
$ 9_{3} $ & 4 & \Cref{thm:main} \\
$ 9_{4} $ & 4 & \Cref{thm:main} \\
$ 9_{5} $ & $[4,5]$ & \cite{rawdon_upper_2002, jin_polygon_1997} \\
$ 9_{6} $ & 4 & \Cref{thm:main} \\
$ 9_{7} $ & 4 & \cite{shonkwiler_new_2020} \\
$ 9_{8} $ & $[4,5]$ & \cite{rawdon_upper_2002, jin_polygon_1997} \\
$ 9_{9} $ & 4 & \Cref{thm:main} \\
$ 9_{10} $ & $[4,5]$ & \cite{rawdon_upper_2002, jin_polygon_1997} \\
$ 9_{11} $ & 4 & \Cref{thm:main} \\
$ 9_{12} $ & $[4,5]$ & \cite{rawdon_upper_2002, jin_polygon_1997} \\
$ 9_{13} $ & 4 & \Cref{thm:main} \\
$ 9_{14} $ & $[4,5]$ & \cite{rawdon_upper_2002, jin_polygon_1997} \\
$ 9_{15} $ & $[4,5]$ & \cite{TomClay, jin_polygon_1997} \\
$ 9_{16} $ & 4 & \cite{shonkwiler_new_2020} \\
$ 9_{17} $ & 4 & \Cref{thm:main} \\
$ 9_{18} $ & 4 & \Cref{thm:main} \\
$ 9_{19} $ & $[4,5]$ & \cite{rawdon_upper_2002, jin_polygon_1997} \\
$ 9_{20} $ & 4 & \cite{shonkwiler_new_2020} \\
$ 9_{21} $ & $[4,5]$ & \cite{TomClay, jin_polygon_1997} \\
$ 9_{22} $ & 4 & \Cref{thm:main} \\
$ 9_{23} $ & 4 & \Cref{thm:main} \\
$ 9_{24} $ & $[4,5]$ & \cite{rawdon_upper_2002, jin_polygon_1997} \\
$ 9_{25} $ & 4 & \Cref{thm:main} \\
$ 9_{26} $ & 4 & \cite{shonkwiler_new_2020} \\
$ 9_{27} $ & 4 & \Cref{thm:main} \\
$ 9_{28} $ & 4 & \cite{shonkwiler_new_2020} \\
$ 9_{29} $ & 4 & \cite{scharein_interactive_1998, jin_polygon_1997} \\
$ 9_{30} $ & 4 & \Cref{thm:main} \\
$ 9_{31} $ & 4 & \Cref{thm:main} \\
$ 9_{32} $ & 4 & \cite{shonkwiler_new_2020} \\
$ 9_{33} $ & 4 & \cite{shonkwiler_new_2020} \\
$ 9_{34} $ & 4 & \cite{rawdon_upper_2002, jin_polygon_1997} \\
$ 9_{35} $ & 4 & \cite{TomClay, jin_polygon_1997} \\
$ 9_{36} $ & 4 & \Cref{thm:main} \\
$ 9_{37} $ & $[4,5]$ & \cite{rawdon_upper_2002, jin_polygon_1997} \\
$ 9_{38} $ & $[4,5]$ & \cite{rawdon_upper_2002, jin_polygon_1997} \\
$ 9_{39} $ & 4 & \cite{TomClay, jin_polygon_1997} \\
$ 9_{40} $ & 4 & \cite{jin_polygon_1997} \\
$ 9_{41} $ & 4 & \cite{jin_polygon_1997} \\
$ 9_{42} $ & 4 & \cite{jin_polygon_1997} \\
$ 9_{43} $ & 4 & \cite{TomClay, jin_polygon_1997} \\
$ 9_{44} $ & 4 & \cite{rawdon_upper_2002, jin_polygon_1997} \\
$ 9_{45} $ & 4 & \cite{TomClay, jin_polygon_1997} \\
$ 9_{46} $ & 4 & \cite{jin_polygon_1997} \\
$ 9_{47} $ & 4 & \cite{rawdon_upper_2002, jin_polygon_1997} \\
$ 9_{48} $ & 4 & \cite{TomClay, jin_polygon_1997} \\
$ 9_{49} $ & 4 & \cite{rawdon_upper_2002, jin_polygon_1997} \\
\midrule
$ 10_{1} $ & $[4,5]$ & \cite{rawdon_upper_2002, jin_polygon_1997} \\
$ 10_{2} $ & $[4,5]$ & \cite{rawdon_upper_2002, jin_polygon_1997} \\
$ 10_{3} $ & $[4,5]$ & \cite{TomClay, jin_polygon_1997} \\
$ 10_{4} $ & $[4,5]$ & \cite{rawdon_upper_2002, jin_polygon_1997} \\
$ 10_{5} $ & $[4,5]$ & \cite{rawdon_upper_2002, jin_polygon_1997} \\
$ 10_{6} $ & $[4,5]$ & \cite{TomClay, jin_polygon_1997} \\
$ 10_{7} $ & $[4,5]$ & \cite{TomClay, jin_polygon_1997} \\
$ 10_{8} $ & $[4,5]$ & \cite{TomClay, jin_polygon_1997} \\
$ 10_{9} $ & $[4,5]$ & \cite{rawdon_upper_2002, jin_polygon_1997} \\
$ 10_{10} $ & $[4,5]$ & \cite{TomClay, jin_polygon_1997} \\
$ 10_{11} $ & $[4,5]$ & \cite{rawdon_upper_2002, jin_polygon_1997} \\
$ 10_{12} $ & $[4,5]$ & \cite{rawdon_upper_2002, jin_polygon_1997} \\
$ 10_{13} $ & $[4,5]$ & \cite{rawdon_upper_2002, jin_polygon_1997} \\
$ 10_{14} $ & $[4,5]$ & \cite{rawdon_upper_2002, jin_polygon_1997} \\
$ 10_{15} $ & $[4,5]$ & \cite{TomClay, jin_polygon_1997} \\
$ 10_{16} $ & $[4,5]$ & \cite{TomClay, jin_polygon_1997} \\
$ 10_{17} $ & $[4,5]$ & \cite{TomClay, jin_polygon_1997} \\
$ 10_{18} $ & $[4,5]$ & \cite{shonkwiler_all_2021, jin_polygon_1997} \\
$ 10_{19} $ & $[4,5]$ & \cite{rawdon_upper_2002, jin_polygon_1997} \\
$ 10_{20} $ & $[4,5]$ & \cite{TomClay, jin_polygon_1997} \\
$ 10_{21} $ & $[4,5]$ & \cite{TomClay, jin_polygon_1997} \\
$ 10_{22} $ & $[4,5]$ & \cite{TomClay, jin_polygon_1997} \\
$ 10_{23} $ & $[4,5]$ & \cite{TomClay, jin_polygon_1997} \\
$ 10_{24} $ & $[4,5]$ & \cite{TomClay, jin_polygon_1997} \\
$ 10_{25} $ & $[4,5]$ & \cite{rawdon_upper_2002, jin_polygon_1997} \\
$ 10_{26} $ & $[4,5]$ & \cite{TomClay, jin_polygon_1997} \\
$ 10_{27} $ & $[4,5]$ & \cite{rawdon_upper_2002, jin_polygon_1997} \\
$ 10_{28} $ & $[4,5]$ & \cite{TomClay, jin_polygon_1997} \\
$ 10_{29} $ & $[4,5]$ & \cite{rawdon_upper_2002, jin_polygon_1997} \\
$ 10_{30} $ & $[4,5]$ & \cite{TomClay, jin_polygon_1997} \\
$ 10_{31} $ & $[4,5]$ & \cite{TomClay, jin_polygon_1997} \\
$ 10_{32} $ & $[4,5]$ & \cite{rawdon_upper_2002, jin_polygon_1997} \\
$ 10_{33} $ & $[4,5]$ & \cite{rawdon_upper_2002, jin_polygon_1997} \\
$ 10_{34} $ & $[4,5]$ & \cite{TomClay, jin_polygon_1997} \\
$ 10_{35} $ & $[4,5]$ & \cite{TomClay, jin_polygon_1997} \\
$ 10_{36} $ & $[4,5]$ & \cite{rawdon_upper_2002, jin_polygon_1997} \\
$ 10_{37} $ & $[4,5]$ & \cite{Adams:2020vm} \\
$ 10_{38} $ & $[4,5]$ & \cite{TomClay, jin_polygon_1997} \\
$ 10_{39} $ & $[4,5]$ & \cite{TomClay, jin_polygon_1997} \\
$ 10_{40} $ & $[4,5]$ & \cite{rawdon_upper_2002, jin_polygon_1997} \\
$ 10_{41} $ & $[4,5]$ & \cite{rawdon_upper_2002, jin_polygon_1997} \\
$ 10_{42} $ & $[4,5]$ & \cite{rawdon_upper_2002, jin_polygon_1997} \\
$ 10_{43} $ & $[4,5]$ & \cite{TomClay, jin_polygon_1997} \\
$ 10_{44} $ & $[4,5]$ & \cite{TomClay, jin_polygon_1997} \\
$ 10_{45} $ & $[4,5]$ & \cite{rawdon_upper_2002, jin_polygon_1997} \\
$ 10_{46} $ & $[4,5]$ & \cite{TomClay, jin_polygon_1997} \\
$ 10_{47} $ & $[4,5]$ & \cite{TomClay, jin_polygon_1997} \\
$ 10_{48} $ & $[4,5]$ & \cite{rawdon_upper_2002, jin_polygon_1997} \\
$ 10_{49} $ & $[4,5]$ & \cite{rawdon_upper_2002, jin_polygon_1997} \\
$ 10_{50} $ & $[4,5]$ & \cite{TomClay, jin_polygon_1997} \\
$ 10_{51} $ & $[4,5]$ & \cite{TomClay, jin_polygon_1997} \\
$ 10_{52} $ & $[4,5]$ & \cite{rawdon_upper_2002, jin_polygon_1997} \\
$ 10_{53} $ & $[4,5]$ & \cite{TomClay, jin_polygon_1997} \\
$ 10_{54} $ & $[4,5]$ & \cite{TomClay, jin_polygon_1997} \\
$ 10_{55} $ & $[4,5]$ & \cite{TomClay, jin_polygon_1997} \\
$ 10_{56} $ & $[4,5]$ & \cite{TomClay, jin_polygon_1997} \\
$ 10_{57} $ & $[4,5]$ & \cite{TomClay, jin_polygon_1997} \\
$ 10_{58} $ & $[4,5]$ & \cite{shonkwiler_all_2021, jin_polygon_1997} \\
$ 10_{59} $ & $[4,5]$ & \cite{rawdon_upper_2002, jin_polygon_1997} \\
$ 10_{60} $ & $[4,5]$ & \cite{rawdon_upper_2002, jin_polygon_1997} \\
$ 10_{61} $ & $[4,5]$ & \cite{rawdon_upper_2002, jin_polygon_1997} \\
$ 10_{62} $ & $[4,5]$ & \cite{TomClay, jin_polygon_1997} \\
$ 10_{63} $ & $[4,5]$ & \cite{rawdon_upper_2002, jin_polygon_1997} \\
$ 10_{64} $ & $[4,5]$ & \cite{TomClay, jin_polygon_1997} \\
$ 10_{65} $ & $[4,5]$ & \cite{TomClay, jin_polygon_1997} \\
$ 10_{66} $ & $[4,5]$ & \cite{shonkwiler_all_2021, jin_polygon_1997} \\
$ 10_{67} $ & $[4,5]$ & \cite{rawdon_upper_2002, jin_polygon_1997} \\
$ 10_{68} $ & $[4,5]$ & \cite{shonkwiler_all_2021, jin_polygon_1997} \\
$ 10_{69} $ & $[4,5]$ & \cite{rawdon_upper_2002, jin_polygon_1997} \\
$ 10_{70} $ & $[4,5]$ & \cite{TomClay, jin_polygon_1997} \\
$ 10_{71} $ & $[4,5]$ & \cite{TomClay, jin_polygon_1997} \\
$ 10_{72} $ & $[4,5]$ & \cite{TomClay, jin_polygon_1997} \\
$ 10_{73} $ & $[4,5]$ & \cite{TomClay, jin_polygon_1997} \\
$ 10_{74} $ & $[4,5]$ & \cite{TomClay, jin_polygon_1997} \\
$ 10_{75} $ & $[4,5]$ & \cite{TomClay, jin_polygon_1997} \\
$ 10_{76} $ & $[4,5]$ & \cite{shonkwiler_new_2020} \\
$ 10_{77} $ & $[4,5]$ & \cite{TomClay, jin_polygon_1997} \\
$ 10_{78} $ & $[4,5]$ & \cite{TomClay, jin_polygon_1997} \\
$ 10_{79} $ & $[4,5]$ & \cite{jin_polygon_1997, scharein_interactive_1998} \\
$ 10_{80} $ & $[4,5]$ & \cite{shonkwiler_all_2021, jin_polygon_1997} \\
$ 10_{81} $ & $[4,5]$ & \cite{rawdon_upper_2002, jin_polygon_1997} \\
$ 10_{82} $ & $[4,5]$ & \cite{shonkwiler_all_2021, jin_polygon_1997} \\
$ 10_{83} $ & $[4,5]$ & \cite{TomClay, jin_polygon_1997} \\
$ 10_{84} $ & $[4,5]$ & \cite{shonkwiler_all_2021, jin_polygon_1997} \\
$ 10_{85} $ & $[4,5]$ & \cite{TomClay, jin_polygon_1997} \\
$ 10_{86} $ & $[4,5]$ & \cite{rawdon_upper_2002, jin_polygon_1997} \\
$ 10_{87} $ & $[4,5]$ & \cite{rawdon_upper_2002, jin_polygon_1997} \\
$ 10_{88} $ & $[4,5]$ & \cite{rawdon_upper_2002, jin_polygon_1997} \\
$ 10_{89} $ & $[4,5]$ & \cite{rawdon_upper_2002, jin_polygon_1997} \\
$ 10_{90} $ & $[4,5]$ & \cite{TomClay, jin_polygon_1997} \\
$ 10_{91} $ & $[4,5]$ & \cite{TomClay, jin_polygon_1997} \\
$ 10_{92} $ & $[4,5]$ & \cite{rawdon_upper_2002, jin_polygon_1997} \\
$ 10_{93} $ & $[4,5]$ & \cite{shonkwiler_all_2021, jin_polygon_1997} \\
$ 10_{94} $ & $[4,5]$ & \cite{TomClay, jin_polygon_1997} \\
$ 10_{95} $ & $[4,5]$ & \cite{TomClay, jin_polygon_1997} \\
$ 10_{96} $ & $[4,5]$ & \cite{rawdon_upper_2002, jin_polygon_1997} \\
$ 10_{97} $ & $[4,5]$ & \cite{TomClay, jin_polygon_1997} \\
$ 10_{98} $ & $[4,5]$ & \cite{rawdon_upper_2002, jin_polygon_1997} \\
$ 10_{99} $ & $[4,5]$ & \cite{rawdon_upper_2002, jin_polygon_1997} \\
$ 10_{100} $ & $[4,5]$ & \cite{shonkwiler_all_2021, jin_polygon_1997} \\
$ 10_{101} $ & $[4,5]$ & \cite{TomClay, jin_polygon_1997} \\
$ 10_{102} $ & $[4,5]$ & \cite{rawdon_upper_2002, jin_polygon_1997} \\
$ 10_{103} $ & $[4,5]$ & \cite{TomClay, jin_polygon_1997} \\
$ 10_{104} $ & $[4,5]$ & \cite{rawdon_upper_2002, jin_polygon_1997} \\
$ 10_{105} $ & $[4,5]$ & \cite{TomClay, jin_polygon_1997} \\
$ 10_{106} $ & $[4,5]$ & \cite{TomClay, jin_polygon_1997} \\
$ 10_{107} $ & $[4,5]$ & \cite{scharein_interactive_1998, jin_polygon_1997} \\
$ 10_{108} $ & $[4,5]$ & \cite{rawdon_upper_2002, jin_polygon_1997} \\
$ 10_{109} $ & $[4,5]$ & \cite{rawdon_upper_2002, jin_polygon_1997} \\
$ 10_{110} $ & $[4,5]$ & \cite{TomClay, jin_polygon_1997} \\
$ 10_{111} $ & $[4,5]$ & \cite{TomClay, jin_polygon_1997} \\
$ 10_{112} $ & $[4,5]$ & \cite{TomClay, jin_polygon_1997} \\
$ 10_{113} $ & $[4,5]$ & \cite{rawdon_upper_2002, jin_polygon_1997} \\
$ 10_{114} $ & $[4,5]$ & \cite{rawdon_upper_2002, jin_polygon_1997} \\
$ 10_{115} $ & $[4,5]$ & \cite{TomClay, jin_polygon_1997} \\
$ 10_{116} $ & $[4,5]$ & \cite{rawdon_upper_2002, jin_polygon_1997} \\
$ 10_{117} $ & $[4,5]$ & \cite{TomClay, jin_polygon_1997} \\
$ 10_{118} $ & $[4,5]$ & \cite{TomClay, jin_polygon_1997} \\
$ 10_{119} $ & $[4,5]$ & \cite{scharein_interactive_1998, jin_polygon_1997} \\
$ 10_{120} $ & $[4,5]$ & \cite{rawdon_upper_2002, jin_polygon_1997} \\
$ 10_{121} $ & $[4,5]$ & \cite{rawdon_upper_2002, jin_polygon_1997} \\
$ 10_{122} $ & $[4,5]$ & \cite{rawdon_upper_2002, jin_polygon_1997} \\
$ 10_{123} $ & $[4,5]$ & \cite{rawdon_upper_2002, jin_polygon_1997} \\
$ 10_{124} $ & 5 & \cite{kuiper_new_1987} \\
$ 10_{125} $ & $[4,5]$ & \cite{rawdon_upper_2002, jin_polygon_1997} \\
$ 10_{126} $ & $[4,5]$ & \cite{TomClay, jin_polygon_1997} \\
$ 10_{127} $ & $[4,5]$ & \cite{rawdon_upper_2002, jin_polygon_1997} \\
$ 10_{128} $ & $[4,5]$ & \cite{rawdon_upper_2002, jin_polygon_1997} \\
$ 10_{129} $ & $[4,5]$ & \cite{rawdon_upper_2002, jin_polygon_1997} \\
$ 10_{130} $ & $[4,5]$ & \cite{rawdon_upper_2002, jin_polygon_1997} \\
$ 10_{131} $ & $[4,5]$ & \cite{TomClay, jin_polygon_1997} \\
$ 10_{132} $ & $[4,5]$ & \cite{rawdon_upper_2002, jin_polygon_1997} \\
$ 10_{133} $ & $[4,5]$ & \cite{TomClay, jin_polygon_1997} \\
$ 10_{134} $ & $[4,5]$ & \cite{rawdon_upper_2002, jin_polygon_1997} \\
$ 10_{135} $ & $[4,5]$ & \cite{rawdon_upper_2002, jin_polygon_1997} \\
$ 10_{136} $ & $[4,5]$ & \cite{rawdon_upper_2002, jin_polygon_1997} \\
$ 10_{137} $ & $[4,5]$ & \cite{TomClay, jin_polygon_1997} \\
$ 10_{138} $ & $[4,5]$ & \cite{TomClay, jin_polygon_1997} \\
$ 10_{139} $ & $[4,5]$ & \cite{rawdon_upper_2002, jin_polygon_1997} \\
$ 10_{140} $ & $[4,5]$ & \cite{rawdon_upper_2002, jin_polygon_1997} \\
$ 10_{141} $ & $[4,5]$ & \cite{rawdon_upper_2002, jin_polygon_1997} \\
$ 10_{142} $ & $[4,5]$ & \cite{TomClay, jin_polygon_1997} \\
$ 10_{143} $ & $[4,5]$ & \cite{TomClay, jin_polygon_1997} \\
$ 10_{144} $ & $[4,5]$ & \cite{rawdon_upper_2002, jin_polygon_1997} \\
$ 10_{145} $ & $[4,5]$ & \cite{rawdon_upper_2002, jin_polygon_1997} \\
$ 10_{146} $ & $[4,5]$ & \cite{rawdon_upper_2002, jin_polygon_1997} \\
$ 10_{147} $ & $[4,5]$ & \cite{scharein_interactive_1998, jin_polygon_1997} \\
$ 10_{148} $ & $[4,5]$ & \cite{TomClay, jin_polygon_1997} \\
$ 10_{149} $ & $[4,5]$ & \cite{TomClay, jin_polygon_1997} \\
$ 10_{150} $ & $[4,5]$ & \cite{rawdon_upper_2002, jin_polygon_1997} \\
$ 10_{151} $ & $[4,5]$ & \cite{rawdon_upper_2002, jin_polygon_1997} \\
$ 10_{152} $ & $[4,5]$ & \cite{shonkwiler_all_2021, jin_polygon_1997} \\
$ 10_{153} $ & $[4,5]$ & \cite{TomClay, jin_polygon_1997} \\
$ 10_{154} $ & $[4,5]$ & \cite{rawdon_upper_2002, jin_polygon_1997} \\
$ 10_{155} $ & $[4,5]$ & \cite{rawdon_upper_2002, jin_polygon_1997} \\
$ 10_{156} $ & $[4,5]$ & \cite{rawdon_upper_2002, jin_polygon_1997} \\
$ 10_{157} $ & $[4,5]$ & \cite{rawdon_upper_2002, jin_polygon_1997} \\
$ 10_{158} $ & $[4,5]$ & \cite{rawdon_upper_2002, jin_polygon_1997} \\
$ 10_{159} $ & $[4,5]$ & \cite{rawdon_upper_2002, jin_polygon_1997} \\
$ 10_{160} $ & $[4,5]$ & \cite{rawdon_upper_2002, jin_polygon_1997} \\
$ 10_{161} $ & $[4,5]$ & \cite{rawdon_upper_2002, jin_polygon_1997} \\
$ 10_{162} $ & $[4,5]$ & \cite{rawdon_upper_2002, jin_polygon_1997} \\
$ 10_{163} $ & $[4,5]$ & \cite{rawdon_upper_2002, jin_polygon_1997} \\
$ 10_{164} $ & $[4,5]$ & \cite{TomClay, jin_polygon_1997} \\
$ 10_{165} $ & $[4,5]$ & \cite{rawdon_upper_2002, jin_polygon_1997} \\
\end{supertabular*}
\end{multicols*}

\end{center}

\newpage

\section{Exact Values of Superbridge Index}\label{sec:exact values}
 
The prime knots through 16 crossings for which the exact value of superbridge index is known.


\begin{center}
\begin{multicols*}{3}
\TrickSupertabularIntoMulticols

\tablefirsthead{
	\multicolumn{1}{l}{$K$} & \multicolumn{2}{l}{$\superbridge[K]$} \\
	\midrule
}
\tablehead{
	\multicolumn{1}{l}{$K$} & \multicolumn{2}{l}{$\superbridge[K]$} \\
	\midrule
}
\tablelasttail{\bottomrule}

\begin{supertabular*}{.26\textwidth}{lll} \label{tab:exact stick numbers}
$ \!\!0_{1} $ & 1 &  \\
$ 3_{1} $ & 3 & \cite{kuiper_new_1987} \\
$ 4_{1} $ & 3 & \cite{jin_polygon_1997,randell_elementary_1994} \\
$ 5_{1} $ & 4 & \cite{kuiper_new_1987} \\
$ 7_{1} $ & 4 & \cite{kuiper_new_1987}\\
$ 7_{5} $ & 4 & \cite{jin_polygon_1997, meissen_edge_1998}\\
$ 7_{6} $ & 4 & \cite{jin_polygon_1997, meissen_edge_1998}\\
$ 7_{7} $ & 4 & \cite{jin_polygon_1997, meissen_edge_1998}\\
$ 8_{1} $ & 4 & \cite{shonkwiler_new_2020} \\
$ 8_{2} $ & 4 & \cite{shonkwiler_new_2020} \\
$ 8_{3} $ & 4 & \cite{shonkwiler_new_2020} \\
$ 8_{5} $ & 4 & \cite{shonkwiler_new_2020} \\
$ 8_{6} $ & 4 & \cite{shonkwiler_new_2020} \\
$ 8_{7} $ & 4 & \cite{shonkwiler_new_2020} \\
$ 8_{8} $ & 4 & \cite{shonkwiler_new_2020} \\
$ 8_{10} $ & 4 & \cite{shonkwiler_new_2020} \\
$ 8_{11} $ & 4 & \cite{shonkwiler_new_2020} \\
$ 8_{12} $ & 4 & \cite{shonkwiler_new_2020} \\
$ 8_{13} $ & 4 & \cite{shonkwiler_new_2020} \\
$ 8_{14} $ & 4 & \cite{shonkwiler_new_2020} \\
$ 8_{15} $ & 4 & \cite{shonkwiler_new_2020} \\
$ 8_{16} $ & 4 & \cite{jin_polygon_1997, rawdon_upper_2002} \\
$ 8_{17} $ & 4 & \cite{jin_polygon_1997, rawdon_upper_2002} \\
$ 8_{18} $ & 4 &  \cite{calvo_geometric_2001, jin_polygon_1997, rawdon_upper_2002} \\
$ 8_{19} $ & 4 & \cite{kuiper_new_1987} \\
$ 8_{20} $ & 4 & \cite{jin_polygon_1997, negami_ramsey_1991, meissen_edge_1998} \\
$ 8_{21} $ & 4 & \cite{jin_polygon_1997, meissen_edge_1998}\\
$ 9_{1} $ & 4 & \cite{kuiper_new_1987} \\
$ 9_{3} $ & 4 & \Cref{thm:main} \\
$ 9_{4} $ & 4 & \Cref{thm:main} \\
$ 9_{6} $ & 4 & \Cref{thm:main} \\
$ 9_{7} $ & 4 & \cite{shonkwiler_new_2020} \\
$ 9_{9} $ & 4 & \Cref{thm:main} \\
$ 9_{11} $ & 4 & \Cref{thm:main} \\
$ 9_{13} $ & 4 & \Cref{thm:main} \\
$ 9_{16} $ & 4 & \cite{shonkwiler_new_2020} \\
$ 9_{17} $ & 4 & \Cref{thm:main} \\
$ 9_{18} $ & 4 & \Cref{thm:main} \\
$ 9_{20} $ & 4 & \cite{shonkwiler_new_2020} \\
$ 9_{22} $ & 4 & \Cref{thm:main} \\
$ 9_{23} $ & 4 & \Cref{thm:main} \\
$ 9_{25} $ & 4 & \Cref{thm:main} \\
$ 9_{26} $ & 4 & \cite{shonkwiler_new_2020} \\
$ 9_{27} $ & 4 & \Cref{thm:main} \\
$ 9_{28} $ & 4 & \cite{shonkwiler_new_2020} \\
$ 9_{29} $ & 4 & \cite{jin_polygon_1997, scharein_interactive_1998} \\
$ 9_{30} $ & 4 & \Cref{thm:main} \\
$ 9_{31} $ & 4 & \Cref{thm:main} \\
$ 9_{32} $ & 4 & \cite{shonkwiler_new_2020} \\
$ 9_{33} $ & 4 & \cite{shonkwiler_new_2020} \\
$ 9_{34} $ & 4 & \cite{jin_polygon_1997, rawdon_upper_2002} \\
$ 9_{35} $ & 4 & \cite{TomClay, jin_polygon_1997} \\
$ 9_{36} $ & 4 & \Cref{thm:main} \\
$ 9_{39} $ & 4 & \cite{TomClay, jin_polygon_1997} \\
$ 9_{40} $ & 4 & \cite{jin_polygon_1997} \\
$ 9_{41} $ & 4 & \cite{jin_polygon_1997} \\
$ 9_{42} $ & 4 & \cite{jin_polygon_1997} \\
$ 9_{43} $ & 4 & \cite{TomClay, jin_polygon_1997} \\
$ 9_{44} $ & 4 & \cite{jin_polygon_1997, rawdon_upper_2002} \\
$ 9_{45} $ & 4 & \cite{TomClay, jin_polygon_1997} \\
$ 9_{46} $ & 4 & \cite{jin_polygon_1997} \\
$ 9_{47} $ & 4 & \cite{jin_polygon_1997, rawdon_upper_2002} \\
$ 9_{48} $ & 4 & \cite{TomClay, jin_polygon_1997} \\
$ 9_{49} $ & 4 & \cite{jin_polygon_1997, rawdon_upper_2002} \\
$ 10_{124} $ & 5 & \cite{kuiper_new_1987} \\
$ 11a_{367} $ & 4 & \cite{kuiper_new_1987} \\
$ 11n_{71} $ & 5 & \cite{TomClay, jin_polygon_1997} \\
$ 11n_{72} $ & 5 & \Cref{thm:main} \\
$ 11n_{73} $ & 5 & \cite{TomClay, jin_polygon_1997} \\
$ 11n_{74} $ & 5 & \cite{TomClay, jin_polygon_1997} \\
$ 11n_{75} $ & 5 & \cite{TomClay, jin_polygon_1997} \\
$ 11n_{76} $ & 5 & \cite{TomClay, jin_polygon_1997} \\
$ 11n_{77} $ & 5 & \Cref{thm:main} \\
$ 11n_{78} $ & 5 & \cite{TomClay, jin_polygon_1997} \\
$ 11n_{81} $ & 5 & \cite{TomClay, jin_polygon_1997} \\
$ 12n_{60} $ & 5 & \Cref{thm:main} \\
$ 12n_{66} $ & 5 & \Cref{thm:main} \\
$ 12n_{219} $ & 5 & \Cref{thm:main} \\
$ 12n_{225} $ & 5 & \Cref{thm:main} \\
$ 12n_{553} $ & 5 & \Cref{thm:main} \\
$ 13a_{4878} $ & 4 & \cite{kuiper_new_1987} \\
$ 13n_{226} $ & 5 & \cite{shonkwiler_new_2020} \\
$ 13n_{285} $ & 5 & \cite{blair_knots_2020, jin_polygon_1997} \\
$ 13n_{293} $ & 5 & \cite{blair_knots_2020, jin_polygon_1997} \\
$ 13n_{328} $ & 5 & \cite{shonkwiler_new_2020} \\
$ 13n_{342} $ & 5 & \cite{shonkwiler_new_2020} \\
$ 13n_{343} $ & 5 & \cite{shonkwiler_new_2020} \\
$ 13n_{350} $ & 5 & \cite{shonkwiler_new_2020} \\
$ 13n_{512} $ & 5 & \cite{shonkwiler_new_2020} \\
$ 13n_{587} $ & 5 & \cite{blair_knots_2020, jin_polygon_1997} \\
$ 13n_{592} $ & 5 & \cite{blair_knots_2020, jin_polygon_1997} \\
$ 13n_{607} $ & 5 & \cite{blair_knots_2020, jin_polygon_1997} \\
$ 13n_{611} $ & 5 & \cite{blair_knots_2020, jin_polygon_1997} \\
$ 13n_{835} $ & 5 & \cite{blair_knots_2020, jin_polygon_1997} \\
$ 13n_{973} $ & 5 & \cite{shonkwiler_new_2020} \\
$ 13n_{1177} $ & 5 & \cite{blair_knots_2020, jin_polygon_1997} \\
$ 13n_{1192} $ & 5 & \cite{blair_knots_2020, jin_polygon_1997} \\
$ 13n_{2641} $ & 5 & \cite{shonkwiler_new_2020} \\
$ 13n_{5018} $ & 5 & \cite{shonkwiler_new_2020} \\
$ 14n_{1753} $ & 5 & \cite{shonkwiler_new_2020} \\
$ 14n_{21881}$ & 6 & \cite{kuiper_new_1987} \\
$ 15a_{85263} $ & 4 & \cite{kuiper_new_1987} \\
$ 15n_{41126} $ & 5 & \cite{blair_knots_2020, jin_polygon_1997} \\
$ 15n_{41127} $ & 5 & \cite{blair_knots_2020, jin_polygon_1997} \\
$ 15n_{41185} $ & 5 & \cite{kuiper_new_1987} \\
$ 16n_{783154} $ & 6 & \cite{kuiper_new_1987} \\

\end{supertabular*}
\end{multicols*}

\end{center}

\section{Knot Images and Coordinates} 
\label{sec:coords}

\Crefname{theorem}{Theorem}{Theorems}

This section gives visualizations and coordinates of each of the knot realizations used in the proof of \Cref{thm:main}. The knot coordinates are shown in the three columns to the right of the image and normalized so that the first vertex is at the origin, the second is on the positive $x$-axis, and the third is in the $xy$-plane with positive $y$-coordinate. Each knot is shown in orthographic perspective from the direction of the positive $z$-axis. 

For $11n_{72}$ and $12n_{553}$ the coordinates, in conjunction with \Cref{thm:JinBound}, are enough to certify the desired upper bound on superbridge index. For the remaining knots, the additional certificate is given below the visualization and coordinates: the vector $\vec{u}$ satisfying \eqref{eq:even gordan} for $9_{22}$, $11n_{77}$, $12n_{60}$, and $12n_{219}$ and the matrix $U$ from \Cref{cor:Gordan2} for the remaining knots.

The original floating-point coordinates can be downloaded from the {\tt stick-knot-gen} project~\cite{stick-knot-gen}.

\begin{center}	
	
	\begin{tabular*}{0.75\textwidth}{C{2.2in} R{.4in} R{.4in} R{.4in}}
	\multicolumn{4}{c}{$9_{3}$} \\
	\cline{1-4}\noalign{\smallskip}
	\multirow{12}{*}{\includegraphics[height=1.8in,width=1.8in,keepaspectratio]{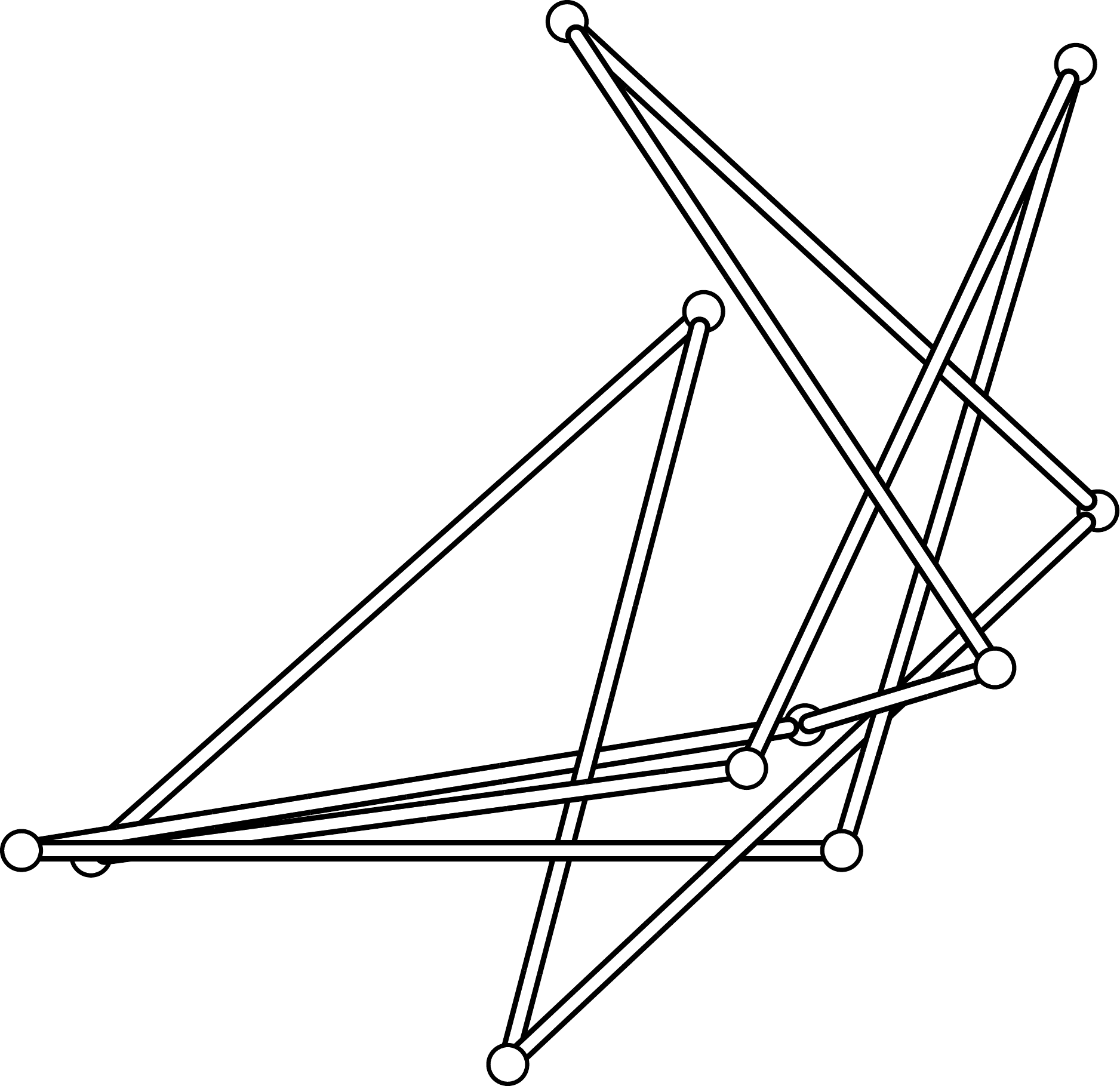}}
	& $0$ & $0$ & $0$ \\
	& $1000$ & $0$ & $0$ \\
	& $1285$ & $958$ & $0$ \\
	& $884$ & $100$ & $320$ \\
	& $85$ & $-7$ & $-271$ \\
	& $832$ & $657$ & $-238$ \\
	& $593$ & $-261$ & $79$ \\
	& $1312$ & $414$ & $-84$ \\
	& $665$ & $1011$ & $390$ \\
	& $1187$ & $223$ & $717$ \\
	& $955$ & $153$ & $-253$\\
	\end{tabular*}
	 \[\begin{tiny}\begin{bmatrix}$1$ & $1$ & $1$ & $1$ & $1$ & $1$ & $1$ \
	& $1$ & $1$ & $1$ & $1$ \\
	$1$ & $1$ & $1$ & $1$ & $1$ & $253078910$ & $1$ & $1$ & $1$ & $1$ & \
	$1$ \\
	$1$ & $1$ & $1$ & $1$ & $1$ & $1$ & $115035555$ & $1$ & $1$ & $1$ & \
	$14393727845$ \\
	$1$ & $657704980$ & $1$ & $418578429$ & $1052504464$ & $376633966$ & \
	$1$ & $1$ & $1$ & $34718493879$ & $1$ \\
	$1240280213$ & $1$ & $417305507$ & $418578431$ & $1$ & $1$ & $1$ & \
	$3652583432$ & $4402732403$ & $1$ & $1$ \\
	$1$ & $844978640$ & $3$ & $1$ & $802489366$ & $1$ & $1$ & $1$ & $1$ & \
	$1$ & $14393727847$ \\
	$1019422757$ & $5$ & $1$ & $2$ & $1$ & $1$ & $1$ & $1$ & $1$ & \
	$34718493881$ & $1$ \\
	$161093877$ & $481696428$ & $71775051$ & $805487149$ & $69082871$ & \
	$156521556$ & $50869993$ & $456789626$ & $216223141$ & $3474445726$ & \
	$609606365$ \\
	$144814783$ & $560343148$ & $197011937$ & $484763007$ & $794003728$ & \
	$204365028$ & $447932175$ & $1210747939$ & $3426142044$ & \
	$1272610542$ & $941105511$ \\
	$2700502673$ & $266444904$ & $774303440$ & $113750373$ & $517843618$ \
	& $267883710$ & $46731239$ & $1840267792$ & $355175436$ & \
	$4614965266$ & $311341360$ \\
	$1505298523$ & $49381178$ & $561386626$ & $165704140$ & $391757632$ & \
	$207156947$ & $500321247$ & $1410752510$ & $1876199764$ & $457419663$ \
	& $1438618133$\end{bmatrix}\end{tiny}\]

	\medskip

	\begin{tabular*}{0.75\textwidth}{C{2.2in} R{.4in} R{.4in} R{.4in}}
	\multicolumn{4}{c}{$9_{4}$} \\
	\cline{1-4}\noalign{\smallskip}
	\multirow{12}{*}{\includegraphics[height=1.8in,width=1.8in,keepaspectratio]{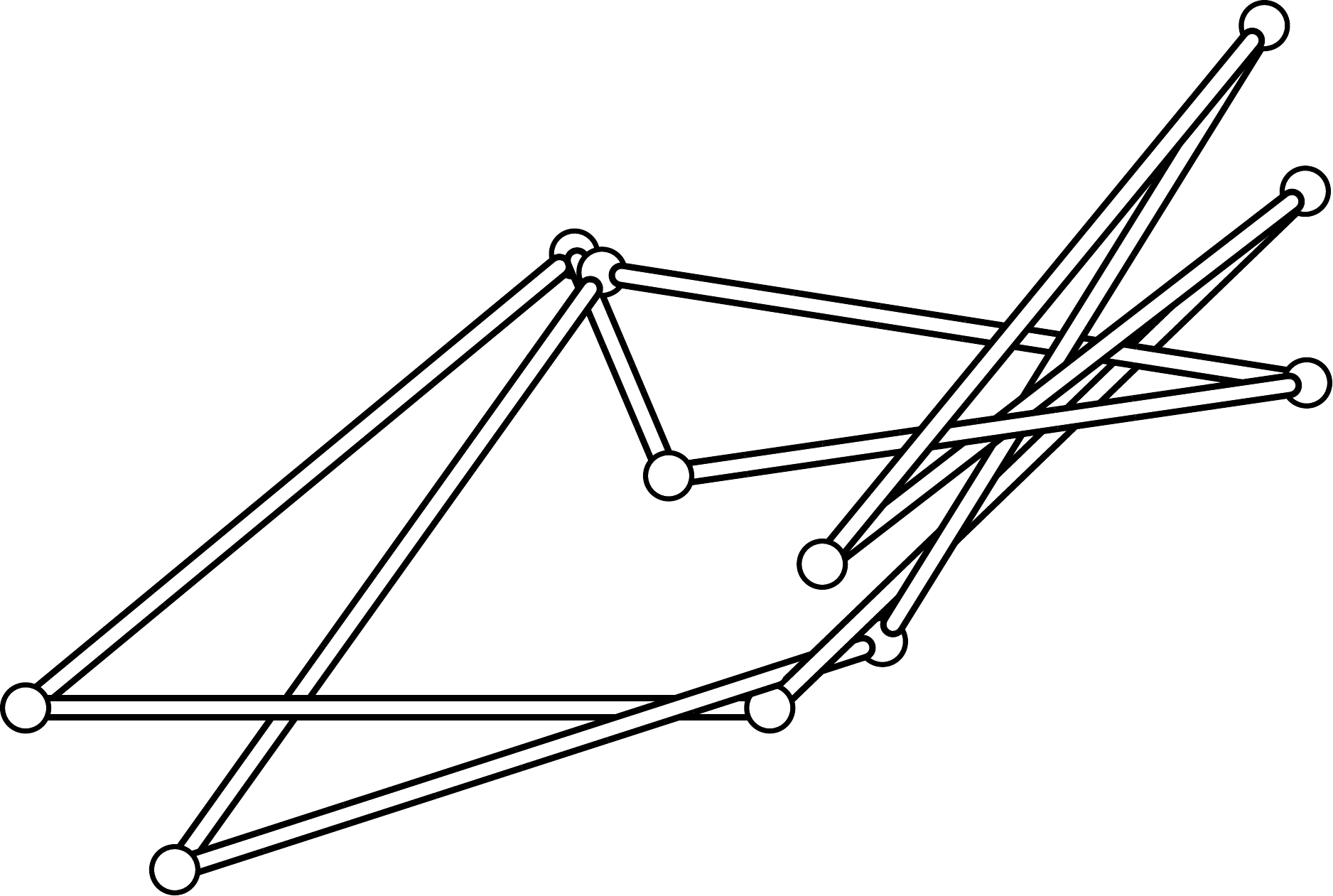}}
	& $0$ & $0$ & $0$ \\
	& $1000$ & $0$ & $0$ \\
	& $1720$ & $694$ & $0$ \\
	& $1071$ & $193$ & $572$ \\
	& $1665$ & $917$ & $221$ \\
	& $1151$ & $89$ & $-6$ \\
	& $200$ & $-217$ & $33$ \\
	& $775$ & $585$ & $-127$ \\
	& $1722$ & $437$ & $157$ \\
	& $864$ & $312$ & $657$ \\
	& $738$ & $610$ & $-289$\\
	\end{tabular*}
	 \[\begin{tiny}\begin{bmatrix}$1$ & $1$ & $1$ & $1$ & $1$ & $1$ & $1$ \
	& $1$ & $1$ & $1$ & $1$ \\
	$1$ & $1$ & $253120999$ & $1$ & $1$ & $1$ & $1$ & $1$ & $1$ & $1$ & \
	$1$ \\
	$1$ & $1$ & $1$ & $1$ & $1$ & $1$ & $152586239$ & $1$ & $1$ & $1$ & \
	$1$ \\
	$1$ & $5387428758$ & $1$ & $1$ & $186051641$ & $1$ & $681384698$ & \
	$1$ & $1$ & $1$ & $1$ \\
	$1$ & $1$ & $1$ & $3$ & $122188121$ & $644758944$ & $1$ & $1$ & \
	$2552652001$ & $1101905682$ & $785500291$ \\
	$412108083$ & $2$ & $1$ & $1$ & $1$ & $1$ & $1$ & $2668475769$ & $1$ \
	& $1$ & $1$ \\
	$1236324265$ & $1$ & $3$ & $28$ & $1$ & $6$ & $528066389$ & $1$ & $1$ \
	& $2038333860$ & $3$ \\
	$15266551$ & $29960368$ & $43047363$ & $7815782$ & $23258145$ & \
	$126807779$ & $46420557$ & $3298398784$ & $655245447$ & $2004553$ & \
	$284780083$ \\
	$207562100$ & $628952294$ & $72634796$ & $19842987$ & $280413639$ & \
	$330174727$ & $24566174$ & $578118108$ & $826423938$ & $1204464626$ & \
	$404153149$ \\
	$737253847$ & $3360035438$ & $396646951$ & $20285775$ & $39377003$ & \
	$41058410$ & $2079635$ & $4867082435$ & $1168804072$ & $161286807$ & \
	$458754813$ \\
	$1464876148$ & $1809746556$ & $78702405$ & $33058272$ & $151532393$ & \
	$185390157$ & $178604983$ & $83268244$ & $1940788369$ & $48127479$ & \
	$1304057157$\end{bmatrix}\end{tiny}\]

	\medskip

	\begin{tabular*}{0.75\textwidth}{C{2.2in} R{.4in} R{.4in} R{.4in}}
	\multicolumn{4}{c}{$9_{6}$} \\
	\cline{1-4}\noalign{\smallskip}
	\multirow{12}{*}{\includegraphics[height=1.8in,width=1.8in,keepaspectratio]{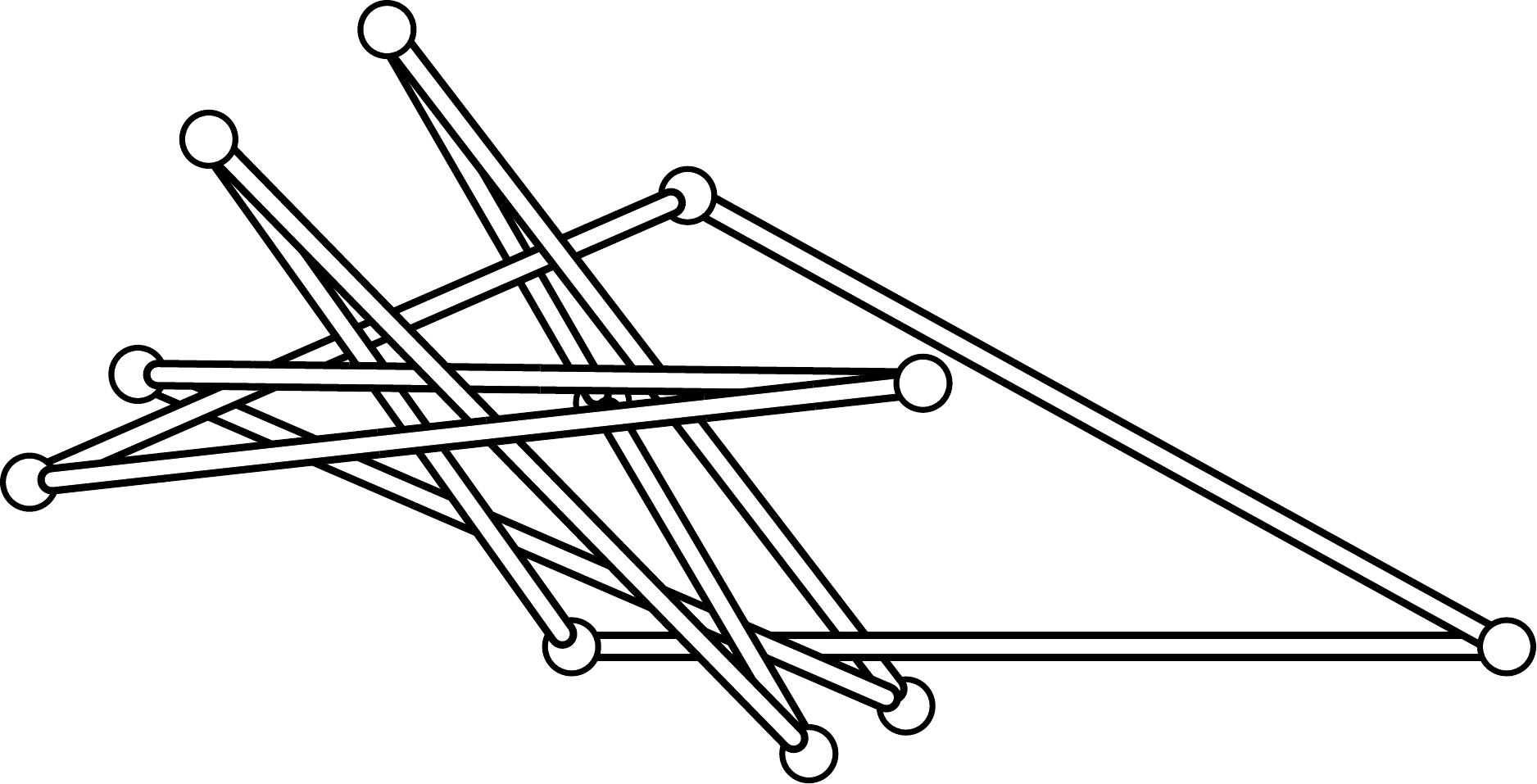}}
	& $0$ & $0$ & $0$ \\
	& $1000$ & $0$ & $0$ \\
	& $124$ & $483$ & $0$ \\
	& $-580$ & $176$ & $641$ \\
	& $376$ & $282$ & $916$ \\
	& $-464$ & $291$ & $373$ \\
	& $358$ & $-63$ & $-74$ \\
	& $-197$ & $660$ & $337$ \\
	& $33$ & $261$ & $-551$ \\
	& $254$ & $-114$ & $350$ \\
	& $-388$ & $543$ & $745$\\
	\end{tabular*}
	 \[\begin{tiny}\begin{bmatrix}$1$ & $1$ & $1$ & $1$ & $1$ & $1$ & $1$ \
	& $1$ & $1$ & $1$ & $1$ \\
	$1$ & $2$ & $1$ & $1$ & $1$ & $1$ & $1$ & $1$ & $1$ & $1112864263$ & \
	$1$ \\
	$2$ & $1$ & $1$ & $1$ & $1$ & $1$ & $1$ & $1$ & $3181145830$ & $1$ & \
	$1$ \\
	$1$ & $1$ & $7986790093$ & $1$ & $40446313$ & $1643510504$ & $1$ & \
	$1$ & $1$ & $1$ & $470098718$ \\
	$1$ & $39162228$ & $1$ & $372465081$ & $20223159$ & $1$ & $233983495$ \
	& $2105861$ & $651976511$ & $1059612083$ & $1$ \\
	$73359287$ & $1$ & $1568714775$ & $186232543$ & $3$ & $1$ & $1$ & $4$ \
	& $2475981364$ & $1$ & $1$ \\
	$32$ & $4$ & $4777752436$ & $3$ & $20223159$ & $1$ & $1$ & $41$ & $1$ \
	& $1$ & $3010168756$ \\
	$24566297$ & $6861048$ & $3862337760$ & $26373238$ & $29977143$ & \
	$772016210$ & $17589227$ & $301023$ & $55104480$ & $16402257$ & \
	$48618876$ \\
	$15491901$ & $17967203$ & $254074348$ & $345511823$ & $13809478$ & \
	$241875406$ & $100606357$ & $1322904$ & $387810364$ & $149231759$ & \
	$24708006$ \\
	$61967286$ & $15123722$ & $171720159$ & $354007729$ & $85119263$ & \
	$82477621$ & $239747297$ & $1371603$ & $254464657$ & $113217341$ & \
	$3471836470$ \\
	$1296156$ & $29089834$ & $107131205$ & $127673557$ & $79207922$ & \
	$2310398106$ & $116060225$ & $1116670$ & $163161428$ & $126899665$ & \
	$265822441$\end{bmatrix}\end{tiny}\]

	\medskip

	\begin{tabular*}{0.75\textwidth}{C{2.2in} R{.4in} R{.4in} R{.4in}}
	\multicolumn{4}{c}{$9_{9}$} \\
	\cline{1-4}\noalign{\smallskip}
	\multirow{12}{*}{\includegraphics[height=1.8in,width=1.8in,keepaspectratio]{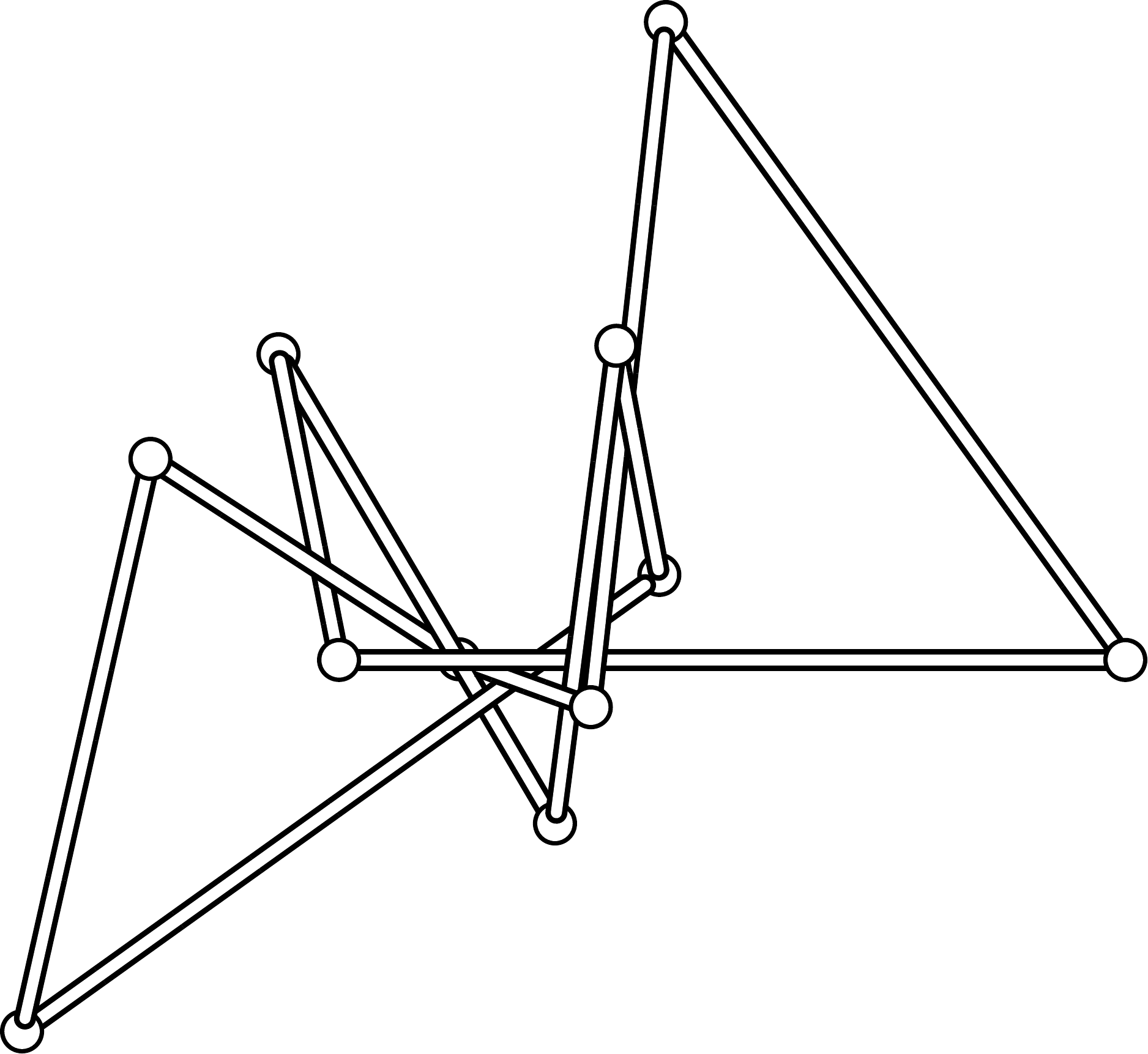}}
	& $0$ & $0$ & $0$ \\
	& $1000$ & $0$ & $0$ \\
	& $416$ & $812$ & $0$ \\
	& $320$ & $-60$ & $481$ \\
	& $153$ & $1$ & $-503$ \\
	& $-240$ & $256$ & $380$ \\
	& $-403$ & $-473$ & $-285$ \\
	& $407$ & $108$ & $-362$ \\
	& $353$ & $400$ & $593$ \\
	& $275$ & $-208$ & $-197$ \\
	& $-77$ & $389$ & $-918$\\
	\end{tabular*}
	 \[\begin{tiny}\begin{bmatrix}$1$ & $1$ & $1$ & $1$ & $1$ & $1$ & $1$ \
	& $1$ & $1$ & $1$ & $1$ \\
	$1$ & $1$ & $1$ & $1$ & $82369555$ & $1$ & $1$ & $12088042132$ & $1$ \
	& $1$ & $1$ \\
	$1$ & $1$ & $1$ & $1$ & $1$ & $1$ & $3617893533$ & $1$ & $1$ & \
	$340470995$ & $1$ \\
	$1280139731$ & $1$ & $485167232$ & $1$ & $1$ & $1599254287$ & $1$ & \
	$1$ & $1$ & $1$ & $1$ \\
	$1$ & $1222246805$ & $338959822$ & $822774430$ & $608229085$ & $1$ & \
	$1$ & $12683277116$ & $486696118$ & $1$ & $4520617633$ \\
	$409077221$ & $5$ & $1$ & $1$ & $3$ & $553545366$ & $3853126796$ & \
	$1$ & $1$ & $3617552792$ & $1$ \\
	$1361329507$ & $6$ & $6$ & $3$ & $1$ & $1$ & $1$ & $1$ & $461787743$ \
	& $2$ & $1$ \\
	$20527753$ & $33291101$ & $127503941$ & $21005683$ & $131493692$ & \
	$232628177$ & $94136669$ & $8530638316$ & $242586007$ & $1626629105$ \
	& $283930053$ \\
	$212668915$ & $76372687$ & $150120702$ & $184229591$ & $113543439$ & \
	$1549831641$ & $2197335797$ & $918626786$ & $29157611$ & $1546137017$ \
	& $600743077$ \\
	$69812503$ & $1482688423$ & $380624447$ & $464206875$ & $633063141$ & \
	$178026749$ & $827882903$ & $103156074$ & $1646260070$ & $1947670452$ \
	& $7136269271$ \\
	$19596715$ & $813474909$ & $507309528$ & $887474901$ & $277868413$ & \
	$964005665$ & $325728759$ & $607309893$ & $976258733$ & $74657653$ & \
	$9140757391$\end{bmatrix}\end{tiny}\]

	\medskip

	\begin{tabular*}{0.75\textwidth}{C{2.2in} R{.4in} R{.4in} R{.4in}}
	\multicolumn{4}{c}{$9_{11}$} \\
	\cline{1-4}\noalign{\smallskip}
	\multirow{12}{*}{\includegraphics[height=1.8in,width=1.8in,keepaspectratio]{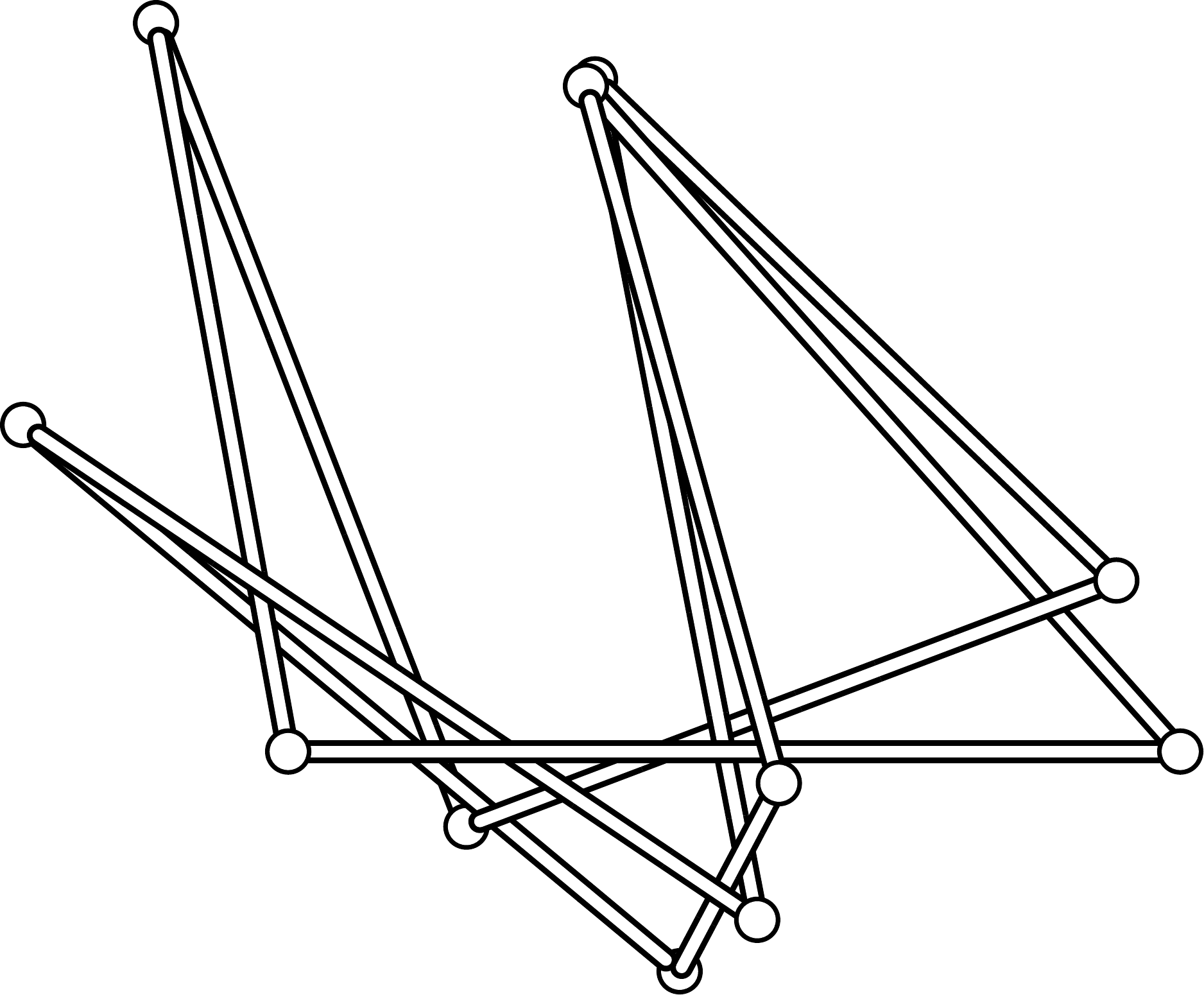}}
	& $0$ & $0$ & $0$ \\
	& $1000$ & $0$ & $0$ \\
	& $334$ & $746$ & $0$ \\
	& $550$ & $-35$ & $586$ \\
	& $439$ & $-246$ & $-386$ \\
	& $-297$ & $366$ & $-97$ \\
	& $526$ & $-189$ & $26$ \\
	& $344$ & $753$ & $-256$ \\
	& $928$ & $192$ & $330$ \\
	& $200$ & $-84$ & $-297$ \\
	& $-148$ & $816$ & $-558$\\
	\end{tabular*}
	 \[\begin{tiny}\begin{bmatrix}$1$ & $1$ & $1$ & $1$ & $1$ & \
	$2368626755$ & $1$ & $1$ & $1$ & $1$ & $1$ \\
	$1$ & $69986829$ & $1$ & $1$ & $4608568067$ & $1$ & $1$ & $1$ & $1$ & \
	$30821056$ & $1$ \\
	$4046685$ & $1$ & $1$ & $1$ & $1$ & $1$ & $1$ & $1$ & $1$ & $1$ & $1$ \
	\\
	$1$ & $3$ & $22114216320$ & $1$ & $1$ & $1$ & $22158061198$ & $1$ & \
	$4153582949$ & $91844549$ & $36384960$ \\
	$1348895$ & $1$ & $1$ & $1$ & $8168076137$ & $2427291954$ & $1$ & \
	$870518840$ & $1$ & $4$ & $1$ \\
	$1348894$ & $126304379$ & $1$ & $673126343$ & $1$ & $1$ & $1$ & $1$ & \
	$5$ & $1$ & $12128320$ \\
	$9969755$ & $6$ & $22675411749$ & $6$ & $1$ & $1$ & $1$ & $5$ & $1$ & \
	$4$ & $7$ \\
	$1141419$ & $14318880$ & $120033110$ & $15770539$ & $736969202$ & \
	$3205461$ & $6055513833$ & $59164979$ & $4170058139$ & $50712235$ & \
	$8554037$ \\
	$6779899$ & $4032158$ & $3968770841$ & $170567743$ & $2030820005$ & \
	$287273843$ & $1413115939$ & $793788443$ & $220881720$ & $47396577$ & \
	$11998081$ \\
	$2511319$ & $1585876$ & $11636038$ & $568931205$ & $63028077$ & \
	$1055081825$ & $363272221$ & $371630189$ & $394308011$ & $17691094$ & \
	$90872208$ \\
	$4828667$ & $93092152$ & $428239011$ & $162250729$ & $1450076601$ & \
	$816658061$ & $21682732901$ & $316274757$ & $1285917521$ & $2503251$ \
	& $80746560$\end{bmatrix}\end{tiny}\]

	\medskip

	\begin{tabular*}{0.75\textwidth}{C{2.2in} R{.4in} R{.4in} R{.4in}}
	\multicolumn{4}{c}{$9_{13}$} \\
	\cline{1-4}\noalign{\smallskip}
	\multirow{12}{*}{\includegraphics[height=1.8in,width=1.8in,keepaspectratio]{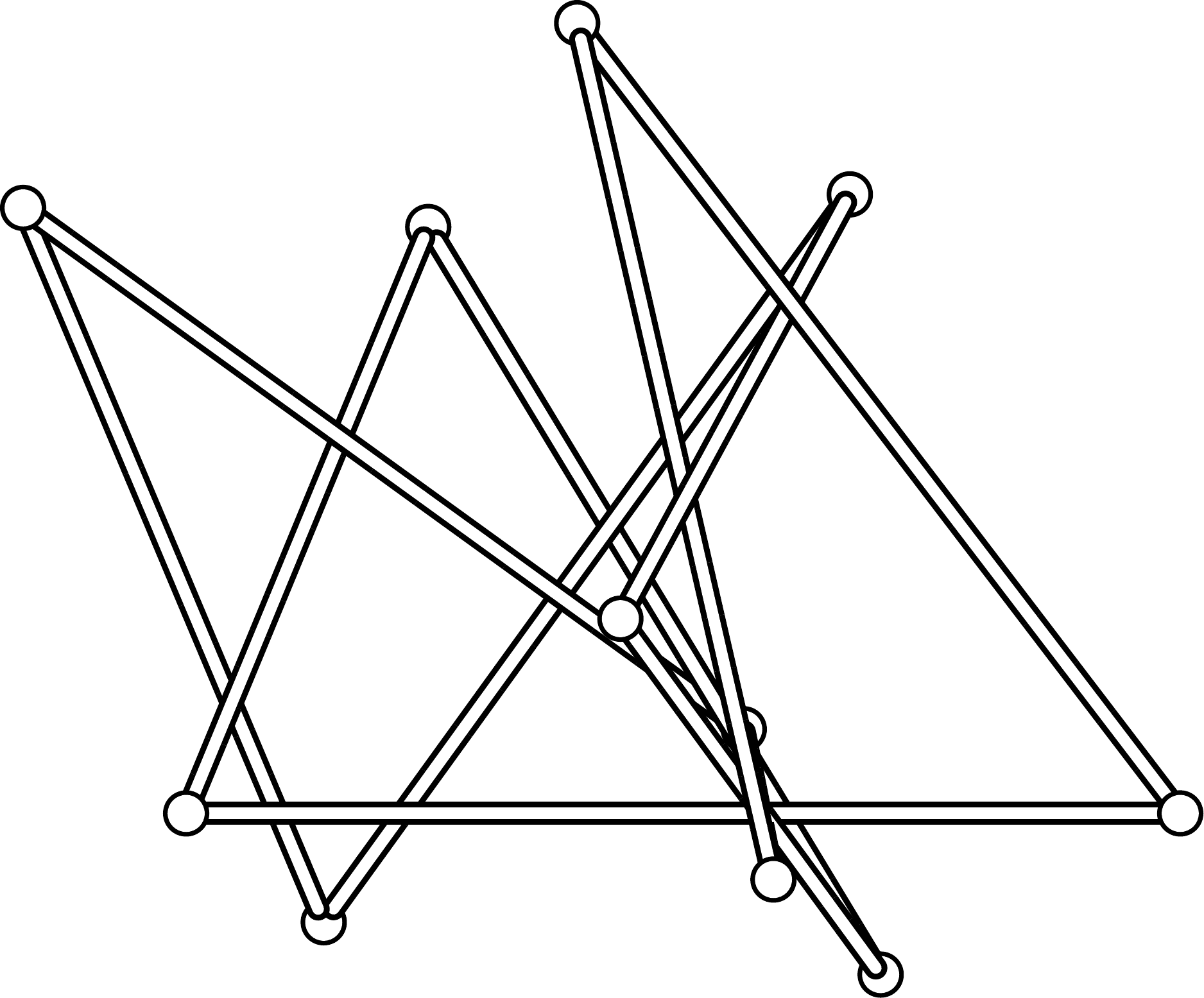}}
	& $0$ & $0$ & $0$ \\
	& $1000$ & $0$ & $0$ \\
	& $393$ & $795$ & $0$ \\
	& $591$ & $-66$ & $468$ \\
	& $561$ & $85$ & $-520$ \\
	& $-164$ & $609$ & $-74$ \\
	& $138$ & $-110$ & $-700$ \\
	& $668$ & $623$ & $-271$ \\
	& $437$ & $196$ & $603$ \\
	& $699$ & $-162$ & $-293$ \\
	& $244$ & $590$ & $-770$\\
	\end{tabular*}
	 \[\begin{tiny}\begin{bmatrix}$1$ & $1$ & $1$ & $1$ & $1$ & $1$ & $1$ \
	& $1$ & $1$ & $1$ & $1$ \\
	$1$ & $1$ & $1$ & $1$ & $1$ & $1$ & $1$ & $12432112632$ & $1$ & $1$ & \
	$1$ \\
	$1$ & $1$ & $1$ & $1$ & $1$ & $1$ & $1$ & $11339684105$ & $1$ & $1$ & \
	$1$ \\
	$1$ & $706806593$ & $36668881$ & $1$ & $1$ & $427052599$ & \
	$1798226748$ & $1$ & $1$ & $475777410$ & $1$ \\
	$1267349485$ & $1$ & $1$ & $3356421$ & $1775635611$ & $1$ & $1$ & $1$ \
	& $11554699858$ & $1$ & $610536151$ \\
	$1$ & $176701650$ & $62185086$ & $9$ & $1$ & $120609105$ & \
	$2745527355$ & $23509882335$ & $1$ & $321874077$ & $610536163$ \\
	$3$ & $353403299$ & $9$ & $36$ & $1$ & $5$ & $1$ & $1$ & $5972539667$ \
	& $10$ & $1$ \\
	$1374953025$ & $12777575$ & $19154908$ & $140941$ & $346309809$ & \
	$102880774$ & $14908368$ & $1114940086$ & $368176352$ & $160006563$ & \
	$48285653$ \\
	$214545333$ & $29260734$ & $1175349$ & $1049985$ & $734912703$ & \
	$320436874$ & $121551897$ & $4360197896$ & $56937195$ & $246871547$ & \
	$925523685$ \\
	$436700012$ & $247137186$ & $66812691$ & $1046909$ & $776204205$ & \
	$66632749$ & $3387707526$ & $736947444$ & $12676105547$ & $184758886$ \
	& $729412089$ \\
	$31438527$ & $541200748$ & $53926501$ & $3539751$ & $591593527$ & \
	$304718771$ & $1876484687$ & $132896953$ & $3072164699$ & $112270009$ \
	& $170069412$\end{bmatrix}\end{tiny}\]

	\medskip

	\begin{tabular*}{0.75\textwidth}{C{2.2in} R{.4in} R{.4in} R{.4in}}
	\multicolumn{4}{c}{$9_{17}$} \\
	\cline{1-4}\noalign{\smallskip}
	\multirow{12}{*}{\includegraphics[height=1.8in,width=1.8in,keepaspectratio]{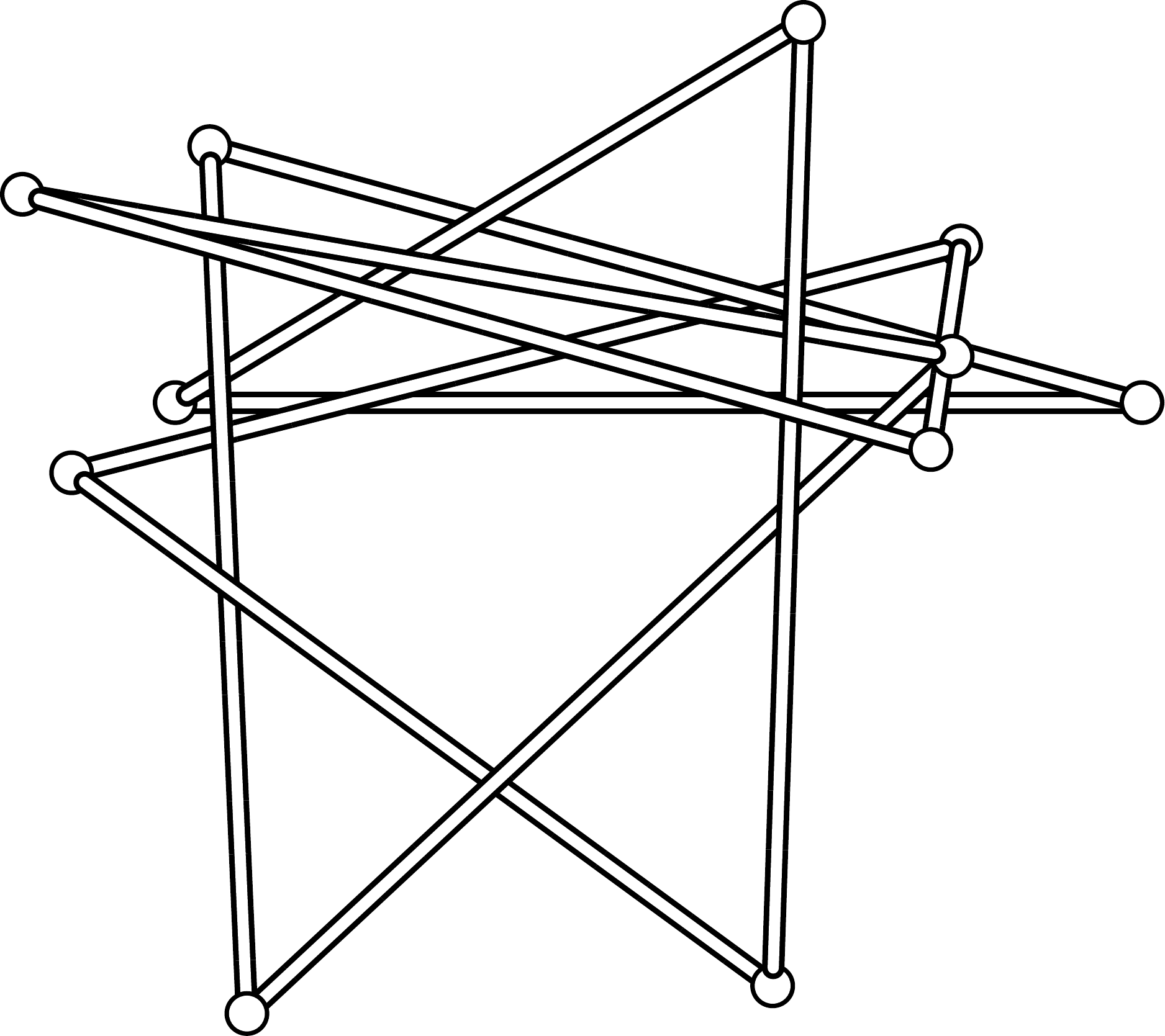}}
	& $0$ & $0$ & $0$ \\
	& $1000$ & $0$ & $0$ \\
	& $36$ & $265$ & $0$ \\
	& $74$ & $-632$ & $440$ \\
	& $804$ & $48$ & $367$ \\
	& $-158$ & $216$ & $584$ \\
	& $782$ & $-49$ & $802$ \\
	& $812$ & $162$ & $-175$ \\
	& $-107$ & $-73$ & $140$ \\
	& $618$ & $-603$ & $579$ \\
	& $650$ & $393$ & $650$\\
	\end{tabular*}
	 \[\begin{tiny}\begin{bmatrix}$1$ & $1$ & $1$ & $1$ & $1$ & $1$ & $1$ \
	& $1$ & $1$ & $1$ & $1$ \\
	$1$ & $1$ & $1$ & $1$ & $1$ & $1$ & $314115327$ & $1$ & $745450626$ & \
	$1$ & $1$ \\
	$1$ & $1600292453$ & $1$ & $1$ & $1$ & $190601031$ & $1$ & $1$ & $1$ \
	& $1$ & $1$ \\
	$51244071966$ & $1$ & $1$ & $1$ & $1262924033$ & $1$ & $1$ & \
	$3567406692$ & $1$ & $1$ & $1$ \\
	$1$ & $1$ & $811883609$ & $3321330367$ & $1$ & $1$ & $1$ & $1$ & \
	$1610556123$ & $1$ & $5768161968$ \\
	$1$ & $1441937750$ & $811883609$ & $1$ & $745359468$ & $1$ & \
	$891553050$ & $1$ & $1$ & $5588584361$ & $1$ \\
	$1$ & $315170260$ & $11$ & $1$ & $1$ & $15$ & $3$ & $1$ & $1$ & $1$ & \
	$3$ \\
	$711452253$ & $146509029$ & $163808559$ & $201344147$ & $273695364$ & \
	$16478531$ & $51636951$ & $146789537$ & $193794429$ & $2345490291$ & \
	$5226457968$ \\
	$1384779802$ & $32324113$ & $80290392$ & $300782365$ & $373632247$ & \
	$44653888$ & $382761166$ & $1525858575$ & $166933119$ & $2461084575$ \
	& $992448746$ \\
	$17436879675$ & $1014438289$ & $1191360764$ & $3717284015$ & \
	$653712347$ & $76576878$ & $236870855$ & $395113987$ & $877438948$ & \
	$4189549685$ & $1727525214$ \\
	$60159279090$ & $1064293944$ & $13186358$ & $2432067637$ & \
	$655203665$ & $259387067$ & $105123646$ & $2180349855$ & $1620622984$ \
	& $209162485$ & $125557548$\end{bmatrix}\end{tiny}\]

	\medskip

	\begin{tabular*}{0.75\textwidth}{C{2.2in} R{.4in} R{.4in} R{.4in}}
	\multicolumn{4}{c}{$9_{18}$} \\
	\cline{1-4}\noalign{\smallskip}
	\multirow{12}{*}{\includegraphics[height=1.8in,width=1.8in,keepaspectratio]{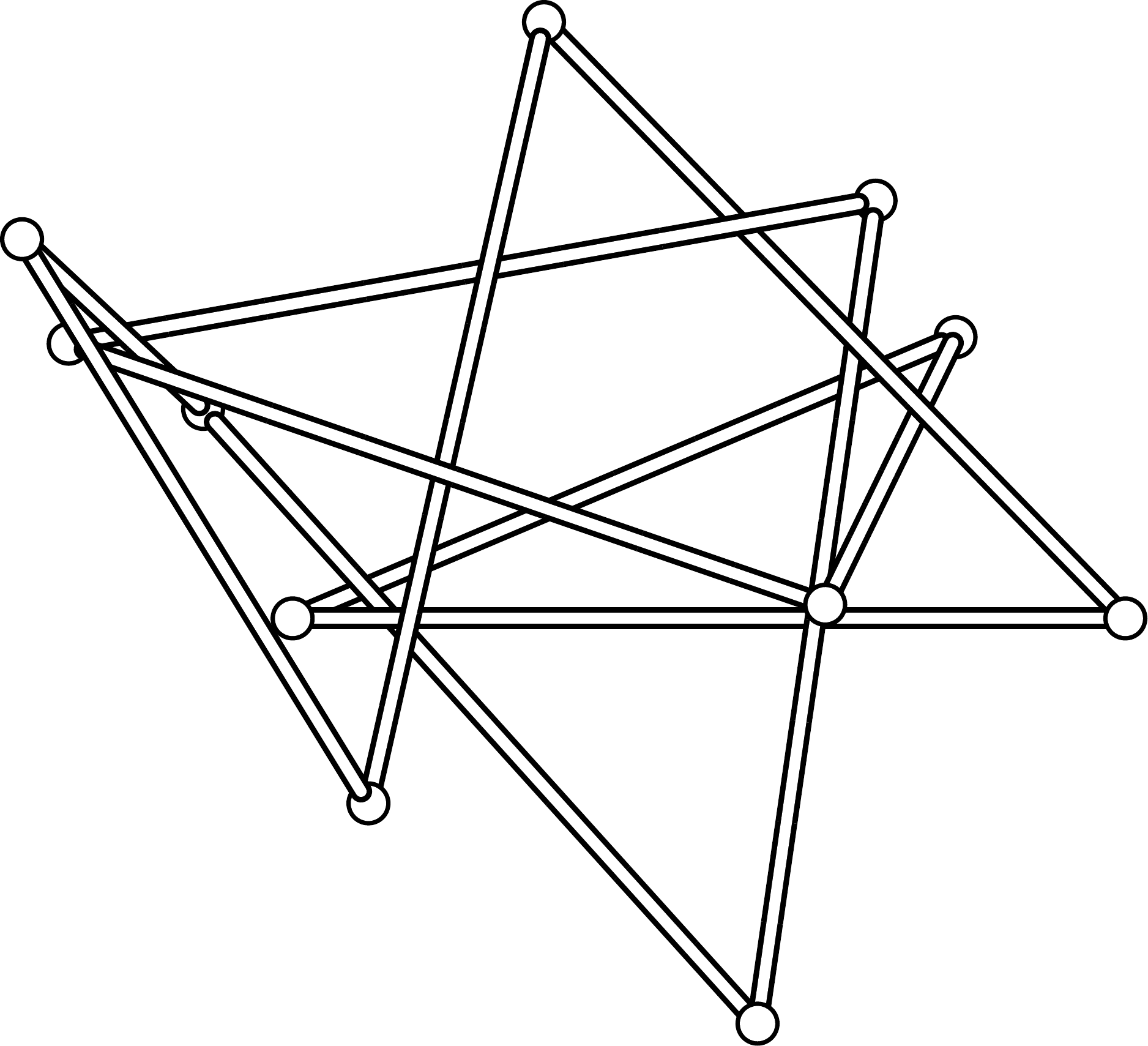}}
	& $0$ & $0$ & $0$ \\
	& $1000$ & $0$ & $0$ \\
	& $302$ & $716$ & $0$ \\
	& $91$ & $-222$ & $275$ \\
	& $-325$ & $455$ & $882$ \\
	& $-107$ & $252$ & $-73$ \\
	& $558$ & $-487$ & $37$ \\
	& $700$ & $502$ & $-23$ \\
	& $-269$ & $330$ & $155$ \\
	& $640$ & $16$ & $432$ \\
	& $796$ & $338$ & $-502$\\
	\end{tabular*}
	 \[\begin{tiny}\begin{bmatrix}$1$ & $1$ & $1$ & $1$ & $1$ & $1$ & \
	$16211286553$ & $1$ & $1$ & $1$ & $1$ \\
	$1$ & $1$ & $30858329$ & $1$ & $1$ & $17762217687$ & $1$ & $1$ & $1$ \
	& $1$ & $1$ \\
	$1$ & $22104539$ & $1$ & $1$ & $1$ & $10564969919$ & $1$ & $1$ & $1$ \
	& $1$ & $728113801$ \\
	$61602923$ & $1$ & $3$ & $1$ & $1$ & $1$ & $1$ & $4024773748$ & $1$ & \
	$74746193045$ & $1$ \\
	$1$ & $22104541$ & $30858331$ & $1$ & $1847148644$ & $1$ & \
	$18892284751$ & $1$ & $1407454648$ & $1$ & $1$ \\
	$54940417$ & $22104541$ & $1$ & $65619096$ & $1$ & $26653970182$ & \
	$1$ & $3832319588$ & $1$ & $1$ & $767409110$ \\
	$11$ & $7$ & $32032131$ & $1$ & $2$ & $5$ & $1$ & $1$ & $1$ & $1$ & \
	$2$ \\
	$11886819$ & $9003228$ & $2839002$ & $2533975$ & $20027306$ & \
	$34502019$ & $1091913663$ & $612386914$ & $304767771$ & $64925682383$ \
	& $282768402$ \\
	$4938585$ & $35780277$ & $2637184$ & $43395440$ & $594214287$ & \
	$221675851$ & $1678374435$ & $4616767263$ & $782809619$ & \
	$3131251157$ & $98095475$ \\
	$108011816$ & $5519263$ & $34323508$ & $46929909$ & $647951814$ & \
	$123826879$ & $5912390181$ & $370257372$ & $897854611$ & $867145469$ \
	& $116125756$ \\
	$54711719$ & $8635998$ & $13161085$ & $79707595$ & $2047261420$ & \
	$74052513$ & $506551563$ & $342704830$ & $529743222$ & $17497321982$ \
	& $400632308$\end{bmatrix}\end{tiny}\]

	\medskip

	\begin{tabular*}{0.75\textwidth}{C{2.2in} R{.4in} R{.4in} R{.4in}}
	\multicolumn{4}{c}{$9_{22}$} \\
	\cline{1-4}\noalign{\smallskip}
	\multirow{12}{*}{\includegraphics[height=1.8in,width=1.8in,keepaspectratio]{9_22.pdf}}
	& $0$ & $0$ & $0$ \\
	& $1000$ & $0$ & $0$ \\
	& $92$ & $419$ & $0$ \\
	& $268$ & $-564$ & $44$ \\
	& $716$ & $266$ & $374$ \\
	& $411$ & $-63$ & $-519$ \\
	& $1029$ & $-280$ & $236$ \\
	& $37$ & $-226$ & $352$ \\
	& $489$ & $10$ & $-509$ \\
	& $879$ & $-326$ & $349$\\
	\end{tabular*}
	\medskip
	\[(107029574,1,1,1,2,97177084,37335363,57495926,18717108,9955666)\]
	
	\medskip

	\begin{tabular*}{0.75\textwidth}{C{2.2in} R{.4in} R{.4in} R{.4in}}
	\multicolumn{4}{c}{$9_{23}$} \\
	\cline{1-4}\noalign{\smallskip}
	\multirow{12}{*}{\includegraphics[height=1.8in,width=1.8in,keepaspectratio]{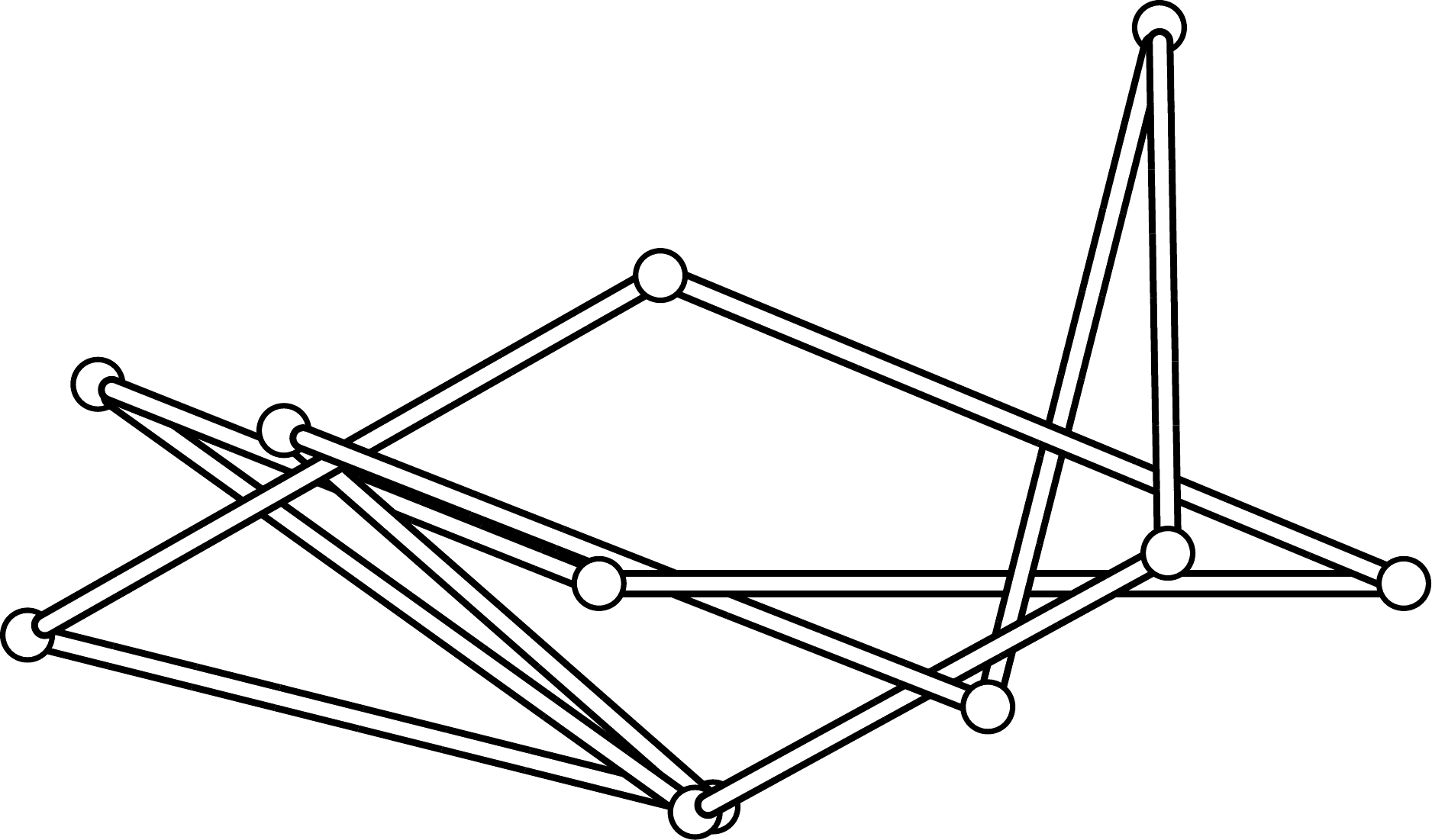}}
	& $0$ & $0$ & $0$ \\
	& $1000$ & $0$ & $0$ \\
	& $76$ & $383$ & $0$ \\
	& $-710$ & $-64$ & $-426$ \\
	& $142$ & $-278$ & $-904$ \\
	& $-392$ & $190$ & $-200$ \\
	& $483$ & $-153$ & $142$ \\
	& $696$ & $691$ & $-349$ \\
	& $707$ & $38$ & $408$ \\
	& $119$ & $-284$ & $-335$ \\
	& $-623$ & $248$ & $-742$ \\
	\end{tabular*}
	 \[\begin{tiny}\begin{bmatrix}$1$ & $1$ & $1$ & $1$ & $1$ & $1$ & $1$ \
	& $1$ & $1$ & $1$ & $1$ \\
	$1$ & $1$ & $1$ & $1$ & $1$ & $1$ & $1$ & $1$ & $1$ & $1$ & $1$ \\
	$113810777$ & $1$ & $1$ & $1$ & $1$ & $1$ & $1$ & $1$ & $1$ & \
	$7776518141$ & $1$ \\
	$1$ & $1$ & $1$ & $594336189$ & $1$ & $8479535543$ & $1$ & $1$ & \
	$30537981917$ & $1$ & $935753973$ \\
	$1$ & $680343456$ & $231960500$ & $1$ & $107828905$ & $1$ & \
	$5139395402$ & $3019840035$ & $1$ & $1$ & $1$ \\
	$1$ & $1$ & $1$ & $896909239$ & $3$ & $1$ & $1$ & $1$ & $15503757477$ \
	& $7501640367$ & $1$ \\
	$119318597$ & $1$ & $234314796$ & $3$ & $107828904$ & $1$ & $1$ & $1$ \
	& $1$ & $1$ & $2$ \\
	$56853036$ & $29567870$ & $124609038$ & $520536059$ & $152277890$ & \
	$2475787461$ & $42803476$ & $27988073$ & $17838739567$ & $248815653$ \
	& $808601845$ \\
	$146069592$ & $433981826$ & $185544858$ & $91242750$ & $6361840$ & \
	$1542757382$ & $680307370$ & $3759013427$ & $1788832557$ & \
	$1483710229$ & $23545428$ \\
	$25221171$ & $1524919025$ & $44185944$ & $62867966$ & $261897966$ & \
	$38201819$ & $4843137239$ & $2250238379$ & $954919009$ & $5402548506$ \
	& $1666174239$ \\
	$7612446$ & $1177159995$ & $28713449$ & $121524435$ & $111634545$ & \
	$10940129573$ & $2356326855$ & $3131764905$ & $391982953$ & \
	$53376286$ & $668217606$\end{bmatrix}\end{tiny}\]

	\medskip

	\begin{tabular*}{0.75\textwidth}{C{2.2in} R{.4in} R{.4in} R{.4in}}
	\multicolumn{4}{c}{$9_{25}$} \\
	\cline{1-4}\noalign{\smallskip}
	\multirow{12}{*}{\includegraphics[height=1.8in,width=1.8in,keepaspectratio]{9_25.pdf}}
	& $0$ & $0$ & $0$ \\
	& $1000$ & $0$ & $0$ \\
	& $478$ & $853$ & $0$ \\
	& $512$ & $-146$ & $26$ \\
	& $-400$ & $249$ & $-92$ \\
	& $562$ & $396$ & $138$ \\
	& $268$ & $96$ & $-769$ \\
	& $-261$ & $438$ & $7$ \\
	& $8$ & $-490$ & $-251$ \\
	& $224$ & $413$ & $-622$ \\
	& $-768$ & $402$ & $-499$ \\
	\end{tabular*}
	 \[\begin{tiny}\begin{bmatrix}$1$ & $1$ & $1$ & $1$ & $1$ & $1$ & $1$ \
	& $1$ & $1$ & $560598895$ & $1$ \\
	$10023997184$ & $1$ & $1$ & $1$ & $1$ & $1$ & $643387175$ & $1$ & \
	$7336018427$ & $653702817$ & $1$ \\
	$1$ & $1$ & $1$ & $1$ & $1$ & $1$ & $1$ & $2343472243$ & $7971644786$ \
	& $1$ & $1688445154$ \\
	$1$ & $1$ & $1$ & $1$ & $1$ & $28666842360$ & $1365051630$ & $1$ & \
	$1$ & $1144406098$ & $1$ \\
	$1$ & $1$ & $1$ & $268087279$ & $1265219757$ & $1$ & $1$ & $1$ & $1$ \
	& $1$ & $1$ \\
	$1$ & $125487117$ & $120699260$ & $1$ & $311867785$ & $1$ & $1$ & $1$ \
	& $1$ & $1$ & $1$ \\
	$6028182429$ & $1$ & $120699263$ & $3$ & $1426407667$ & $1$ & $1$ & \
	$1$ & $1$ & $6$ & $4$ \\
	$4895581470$ & $567322658$ & $102400928$ & $798486118$ & $22035705$ & \
	$194397770$ & $61350220$ & $1120439309$ & $49352735$ & $110472066$ & \
	$1165703994$ \\
	$6760551$ & $74959938$ & $77412490$ & $66008727$ & $26293220$ & \
	$1361890526$ & $123659494$ & $1299711247$ & $2134873307$ & \
	$820466962$ & $582352931$ \\
	$2399354$ & $643239294$ & $303040821$ & $37886430$ & $140201278$ & \
	$31985620272$ & $1413701530$ & $218763732$ & $690714185$ & \
	$131619140$ & $277914126$ \\
	$214770399$ & $108509612$ & $327975748$ & $851232408$ & $74901645$ & \
	$12181335938$ & $1128149711$ & $1693438564$ & $4800658539$ & \
	$11448665$ & $189208084$\end{bmatrix}\end{tiny}\]

	\medskip

	\begin{tabular*}{0.75\textwidth}{C{2.2in} R{.4in} R{.4in} R{.4in}}
	\multicolumn{4}{c}{$9_{27}$} \\
	\cline{1-4}\noalign{\smallskip}
	\multirow{12}{*}{\includegraphics[height=1.8in,width=1.8in,keepaspectratio]{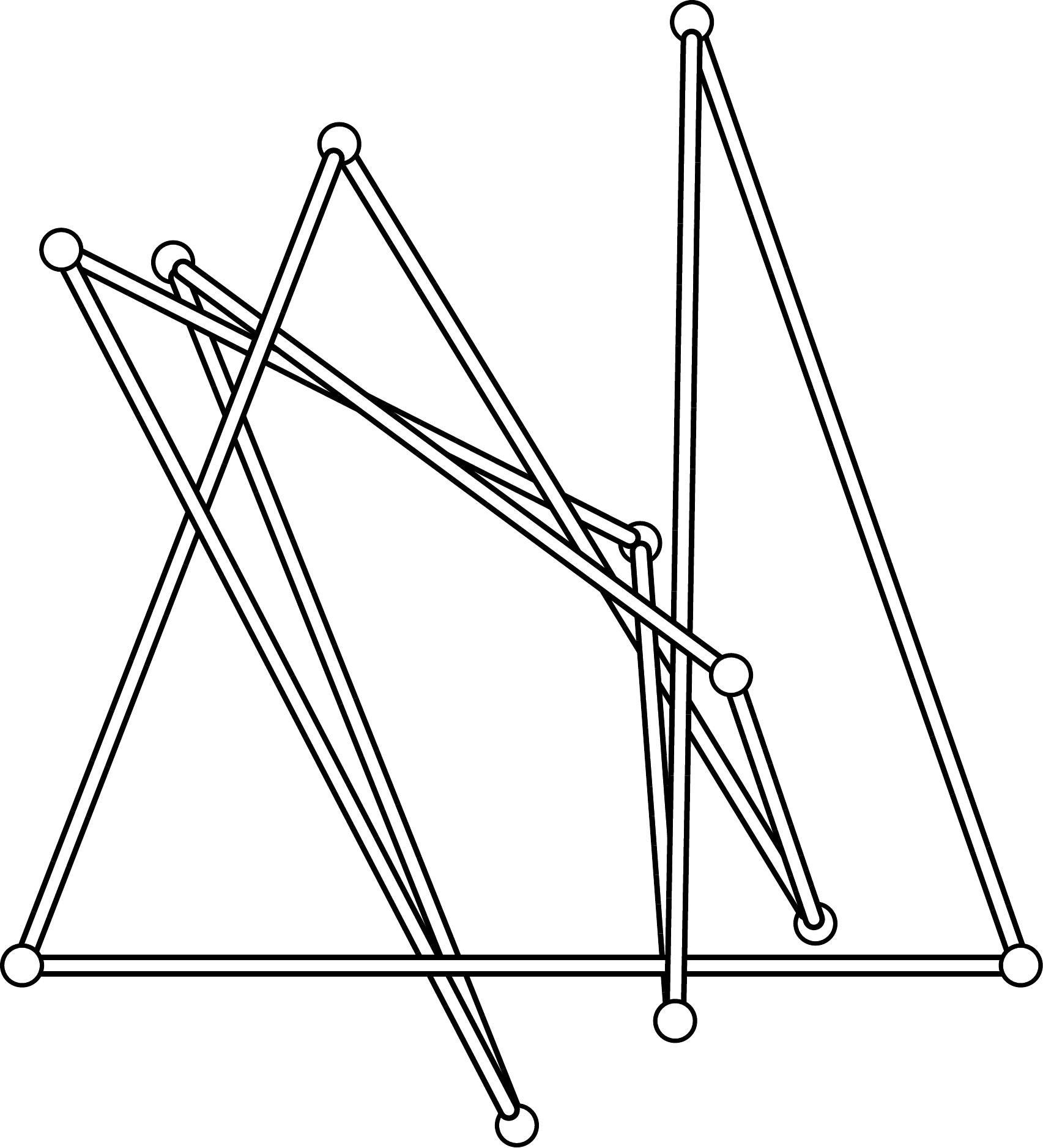}}
	& $0$ & $0$ & $0$ \\
	& $1000$ & $0$ & $0$ \\
	& $671$ & $944$ & $0$ \\
	& $654$ & $-56$ & $8$ \\
	& $619$ & $423$ & $-870$ \\
	& $39$ & $716$ & $-109$ \\
	& $495$ & $-160$ & $-264$ \\
	& $151$ & $704$ & $-632$ \\
	& $710$ & $291$ & $87$ \\
	& $794$ & $42$ & $-877$ \\
	& $318$ & $822$ & $-473$ \\
	\end{tabular*}
	 \[\begin{tiny}\begin{bmatrix}$1$ & $1$ & $262889684$ & $1$ & $1$ & \
	$1$ & $1$ & $1$ & $1$ & $1$ & $1$ \\
	$1$ & $7847390078$ & $1$ & $1$ & $1$ & $1$ & $1$ & $1$ & $1$ & \
	$230952377$ & $1$ \\
	$1$ & $1$ & $1$ & $1$ & $1$ & $1$ & $1$ & $1$ & $873199948$ & $1$ & \
	$1$ \\
	$1$ & $31158284984$ & $1$ & $1$ & $1$ & $1$ & $1280024865$ & \
	$1388695196$ & $1$ & $344675452$ & $1$ \\
	$83053482$ & $1$ & $1$ & $1$ & $5055431027$ & $13805102457$ & $1$ & \
	$1$ & $1413239576$ & $1$ & $15640888756$ \\
	$1$ & $1$ & $4$ & $404554605$ & $1$ & $1$ & $1$ & $2777393120$ & $1$ \
	& $1$ & $1$ \\
	$81$ & $4$ & $575819163$ & $5$ & $1$ & $1$ & $1568405053$ & $1$ & $1$ \
	& $4$ & $1$ \\
	$21769815$ & $29208119282$ & $1068327$ & $73472469$ & $54731667$ & \
	$12281625511$ & $162551761$ & $373062650$ & $291347024$ & $26899937$ \
	& $271603617$ \\
	$35919104$ & $170269626$ & $22441872$ & $384026748$ & $2106979657$ & \
	$224617079$ & $23862425$ & $1356736353$ & $800809090$ & $29063382$ & \
	$2532544202$ \\
	$47831262$ & $129202033$ & $783370605$ & $42382366$ & $1213807343$ & \
	$8531371491$ & $1291553665$ & $1812626691$ & $675753926$ & \
	$158354530$ & $13607358180$ \\
	$34156617$ & $649855006$ & $163560582$ & $186536581$ & $4106403181$ & \
	$1664690609$ & $629983557$ & $209503010$ & $751141194$ & $447257698$ \
	& $7967627966$\end{bmatrix}\end{tiny}\]

	\medskip

	\begin{tabular*}{0.75\textwidth}{C{2.2in} R{.4in} R{.4in} R{.4in}}
	\multicolumn{4}{c}{$9_{30}$} \\
	\cline{1-4}\noalign{\smallskip}
	\multirow{12}{*}{\includegraphics[height=1.8in,width=1.8in,keepaspectratio]{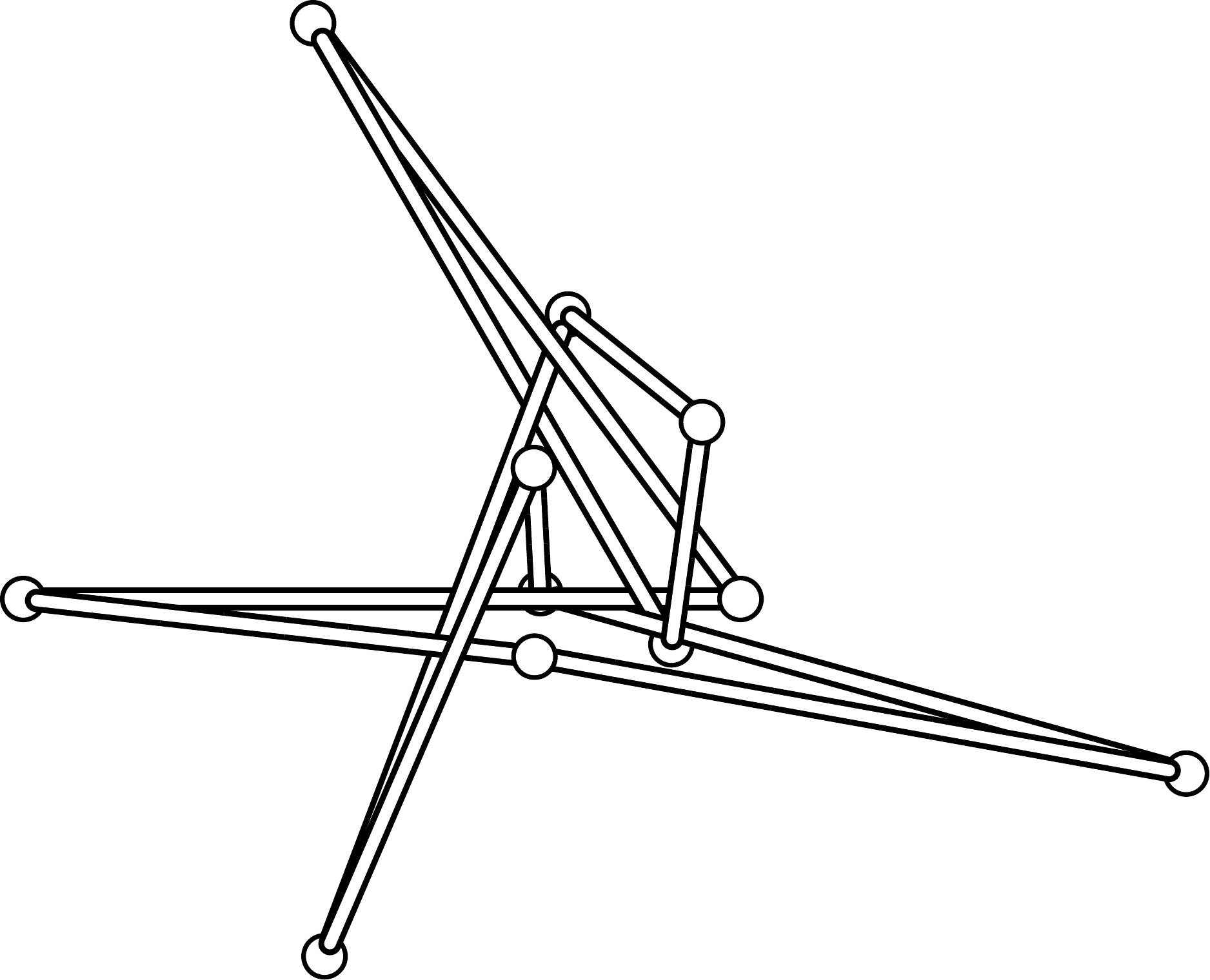}}
	& $0$ & $0$ & $0$ \\
	& $1000$ & $0$ & $0$ \\
	& $404$ & $803$ & $0$ \\
	& $904$ & $-63$ & $2$ \\
	& $946$ & $247$ & $952$ \\
	& $759$ & $397$ & $-19$ \\
	& $420$ & $-498$ & $270$ \\
	& $712$ & $183$ & $942$ \\
	& $721$ & $8$ & $-43$ \\
	& $1621$ & $-243$ & $312$ \\
	& $713$ & $-80$ & $697$ \\
	\end{tabular*}
	 \[\begin{tiny}\begin{bmatrix}$1$ & $1$ & $1$ & $1$ & $1$ & $1$ & $1$ \
	& $1$ & $1$ & $1$ & $1$ \\
	$12437473$ & $1$ & $1$ & $1$ & $1$ & $1$ & $1$ & $577487407$ & $1$ & \
	$1$ & $1$ \\
	$1$ & $1$ & $1$ & $1$ & $1$ & $1$ & $260639873$ & $1$ & $1$ & $1$ & \
	$1$ \\
	$1$ & $1$ & $1$ & $3310686559$ & $1$ & $421434401$ & $1$ & $1$ & \
	$6368992901$ & $1$ & $1207138066$ \\
	$3$ & $1406257395$ & $24702488$ & $1$ & $8178655003$ & $1$ & $1$ & \
	$1$ & $1$ & $1$ & $703355439$ \\
	$1$ & $1$ & $1$ & $3735581387$ & $1$ & $380751573$ & $282519864$ & \
	$670519498$ & $1$ & $25064316807$ & $4$ \\
	$77477129$ & $1$ & $49404989$ & $3$ & $7389135180$ & $7$ & $2$ & $1$ \
	& $1$ & $4$ & $1$ \\
	$8161120$ & $744540803$ & $7732102$ & $2850370473$ & $32838207$ & \
	$51224612$ & $5804151$ & $60698763$ & $1914312628$ & $43894377$ & \
	$1692119784$ \\
	$47463080$ & $944557761$ & $31229795$ & $534698731$ & $303461273$ & \
	$86414122$ & $3214384$ & $421910632$ & $2643878907$ & $825827562$ & \
	$30561789$ \\
	$87465935$ & $1090002507$ & $46067476$ & $224985073$ & $3350742742$ & \
	$86138733$ & $234567480$ & $191288209$ & $843984412$ & $34590412498$ \
	& $10989181$ \\
	$6400283$ & $1135573793$ & $39123181$ & $1918292057$ & $34826475$ & \
	$11693541$ & $35597081$ & $173958066$ & $5788098369$ & $29129657180$ \
	& $224365433$\end{bmatrix}\end{tiny}\]

	\medskip

	\begin{tabular*}{0.75\textwidth}{C{2.2in} R{.4in} R{.4in} R{.4in}}
	\multicolumn{4}{c}{$9_{31}$} \\
	\cline{1-4}\noalign{\smallskip}
	\multirow{12}{*}{\includegraphics[height=1.8in,width=1.8in,keepaspectratio]{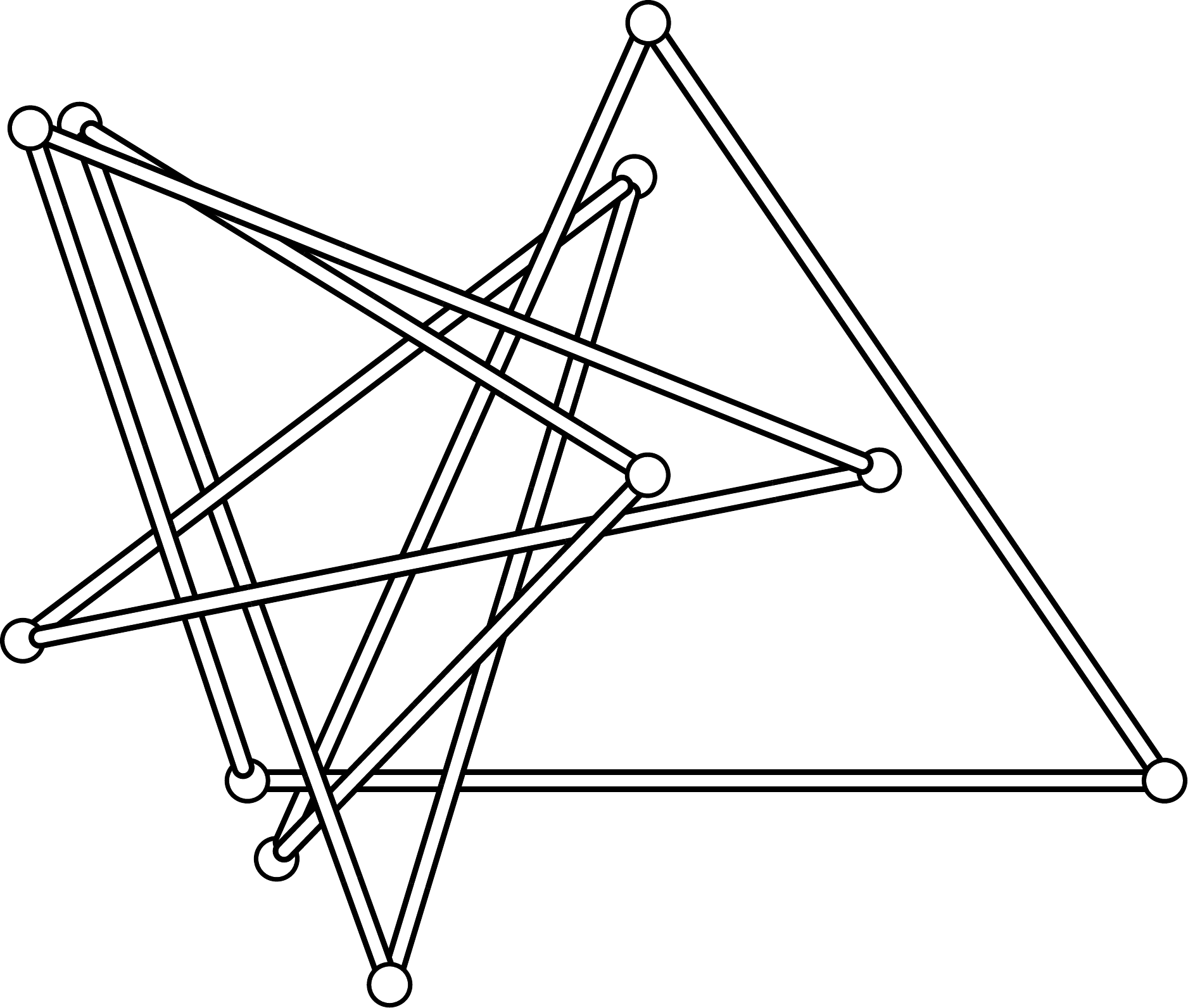}}
	& $0$ & $0$ & $0$ \\
	& $1000$ & $0$ & $0$ \\
	& $437$ & $826$ & $0$ \\
	& $32$ & $-85$ & $-69$ \\
	& $436$ & $333$ & $744$ \\
	& $-183$ & $715$ & $58$ \\
	& $155$ & $-223$ & $135$ \\
	& $422$ & $658$ & $-255$ \\
	& $-245$ & $153$ & $292$ \\
	& $689$ & $338$ & $598$ \\
	& $-237$ & $712$ & $662$ \\
	\end{tabular*}
	 \[\begin{tiny}\begin{bmatrix}$1$ & $1$ & $1$ & $1$ & $1$ & $1$ & $1$ \
	& $1$ & $1$ & $1$ & $1$ \\
	$491876281$ & $1$ & $1$ & $1$ & $1$ & $1$ & $1$ & $1$ & $3354528574$ \
	& $1$ & $1$ \\
	$1$ & $1$ & $1$ & $1$ & $1$ & $1$ & $1$ & $1$ & $2525361769$ & \
	$158087629$ & $46864636843$ \\
	$1$ & $25510362$ & $1$ & $1$ & $1$ & $1$ & $1$ & $11242255527$ & $1$ \
	& $1$ & $1$ \\
	$454261886$ & $1$ & $1$ & $5284164485$ & $1$ & $1813691147$ & \
	$2885517387$ & $1$ & $6821899449$ & $1$ & $1$ \\
	$1$ & $25510364$ & $59543495$ & $1$ & $5803156562$ & $925296372$ & \
	$1$ & $1$ & $1$ & $1$ & $1$ \\
	$8$ & $127551842$ & $89315266$ & $2442529200$ & $2468012324$ & $1$ & \
	$1$ & $1$ & $1$ & $9$ & $1$ \\
	$258214242$ & $12180665$ & $10277724$ & $5054439091$ & $288241295$ & \
	$49962721$ & $134140596$ & $7144416519$ & $666443081$ & $11291115$ & \
	$36313021745$ \\
	$48857123$ & $33735396$ & $35752426$ & $857072468$ & $102767719$ & \
	$458418730$ & $1543535133$ & $1470609912$ & $3230572056$ & $56332267$ \
	& $2577828447$ \\
	$10527292$ & $168498480$ & $127077814$ & $61184503$ & $6816702585$ & \
	$1391694775$ & $2754689006$ & $999523140$ & $742278861$ & $70868761$ \
	& $374408352$ \\
	$63841640$ & $38539193$ & $82612705$ & $5192781$ & $2488892900$ & \
	$1690661440$ & $3538544765$ & $9879562560$ & $106931688$ & \
	$211428838$ & $33734183487$\end{bmatrix}\end{tiny}\]

	\medskip

	\begin{tabular*}{0.75\textwidth}{C{2.2in} R{.4in} R{.4in} R{.4in}}
	\multicolumn{4}{c}{$9_{36}$} \\
	\cline{1-4}\noalign{\smallskip}
	\multirow{12}{*}{\includegraphics[height=1.8in,width=1.8in,keepaspectratio]{9_36.pdf}}
	& $0$ & $0$ & $0$ \\
	& $1000$ & $0$ & $0$ \\
	& $567$ & $902$ & $0$ \\
	& $-185$ & $294$ & $-256$ \\
	& $759$ & $245$ & $71$ \\
	& $318$ & $816$ & $764$ \\
	& $646$ & $-13$ & $311$ \\
	& $718$ & $908$ & $-72$ \\
	& $-183$ & $662$ & $285$ \\
	& $718$ & $517$ & $694$ \\
	& $406$ & $897$ & $-177$ \\
	\end{tabular*}
	 \[\begin{tiny}\begin{bmatrix}$1$ & $1$ & $1$ & $1$ & $1$ & $1$ & $1$ \
	& $1$ & $1$ & $1$ & $1$ \\
	$1$ & $1$ & $1$ & $1$ & $1$ & $1$ & $1$ & $12523027847$ & $1$ & $1$ & \
	$1$ \\
	$5426124359$ & $1$ & $1$ & $1$ & $1424202531$ & $1$ & $1$ & $1$ & $1$ \
	& $6584192388$ & $8992884808$ \\
	$1$ & $1$ & $1$ & $12223967$ & $1$ & $1$ & $263408915$ & \
	$29465794187$ & $9986903746$ & $1$ & $1$ \\
	$495044898$ & $5412091732$ & $100403683$ & $1$ & $3$ & $212199095$ & \
	$1$ & $1$ & $1$ & $1$ & $1$ \\
	$1$ & $1$ & $1$ & $3$ & $2159516409$ & $1$ & $65852229$ & $1$ & $1$ & \
	$6627781141$ & $1$ \\
	$1$ & $1$ & $38906659$ & $19412935$ & $1$ & $53049774$ & $39336205$ & \
	$1$ & $1$ & $1$ & $1$ \\
	$1125551611$ & $1735612295$ & $29180733$ & $858339$ & $1557767516$ & \
	$11511945$ & $88068805$ & $355250231$ & $6884251289$ & $36701288$ & \
	$398991934$ \\
	$1241803067$ & $3170128948$ & $69866228$ & $11625485$ & $267607237$ & \
	$79522441$ & $20150485$ & $2006088206$ & $2575693266$ & $1056878715$ \
	& $2505218830$ \\
	$45910998$ & $355246092$ & $26475285$ & $6280015$ & $89877512$ & \
	$374500950$ & $1677801$ & $145553542$ & $1448979271$ & $824513076$ & \
	$3204959462$ \\
	$3901711428$ & $1642348279$ & $12533499$ & $2078963$ & $40849202$ & \
	$426090556$ & $184771845$ & $27474391863$ & $3948141106$ & \
	$5211154861$ & $7748788486$\end{bmatrix}\end{tiny}\]

	\medskip

	\begin{tabular*}{0.75\textwidth}{C{2.2in} R{.4in} R{.4in} R{.4in}}
	\multicolumn{4}{c}{$11n_{72}$} \\
	\cline{1-4}\noalign{\smallskip}
	\multirow{12}{*}{\includegraphics[height=1.8in,width=1.8in,keepaspectratio]{K11n72.pdf}}
	& $0$ & $0$ & $0$ \\
	& $1000$ & $0$ & $0$ \\
	& $289$ & $703$ & $0$ \\
	& $920$ & $-69$ & $74$ \\
	& $166$ & $135$ & $-550$ \\
	& $858$ & $408$ & $118$ \\
	& $211$ & $-313$ & $368$ \\
	& $359$ & $248$ & $-446$ \\
	& $865$ & $433$ & $396$ \\
	& $589$ & $-346$ & $-167$ \\
	& $805$ & $549$ & $225$\\
	\end{tabular*}

	\medskip

	\begin{tabular*}{0.75\textwidth}{C{2.2in} R{.4in} R{.4in} R{.4in}}
	\multicolumn{4}{c}{$11n_{77}$} \\
	\cline{1-4}\noalign{\smallskip}
	\multirow{12}{*}{\includegraphics[height=1.8in,width=1.8in,keepaspectratio]{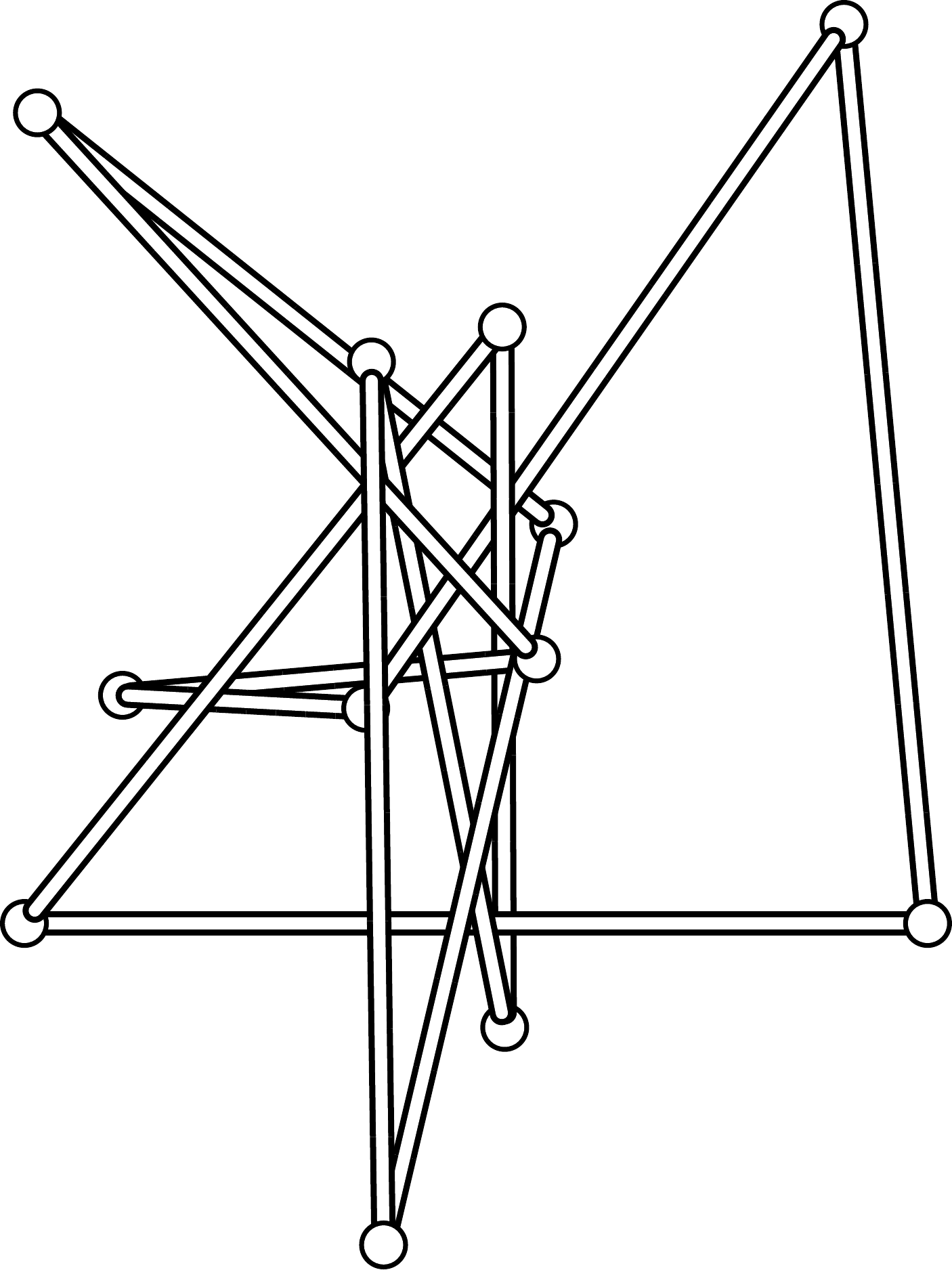}}
	& $0$ & $0$ & $0$ \\
	& $1000$ & $0$ & $0$ \\
	& $908$ & $996$ & $0$ \\
	& $378$ & $238$ & $382$ \\
	& $109$ & $253$ & $-581$ \\
	& $567$ & $293$ & $307$ \\
	& $14$ & $898$ & $880$ \\
	& $586$ & $442$ & $198$ \\
	& $398$ & $-356$ & $770$ \\
	& $384$ & $622$ & $561$ \\
	& $532$ & $-115$ & $-98$ \\
	& $529$ & $661$ & $533$ \\
	\end{tabular*}
	\[(1,1,4360454070,1,1,1,4129928398,1,3083050732,480679674,2231790712,131796380)\]

	\medskip

	\begin{tabular*}{0.75\textwidth}{C{2.2in} R{.4in} R{.4in} R{.4in}}
	\multicolumn{4}{c}{$12n_{60}$} \\
	\cline{1-4}\noalign{\smallskip}
	\multirow{12}{*}{\includegraphics[height=1.8in,width=1.8in,keepaspectratio]{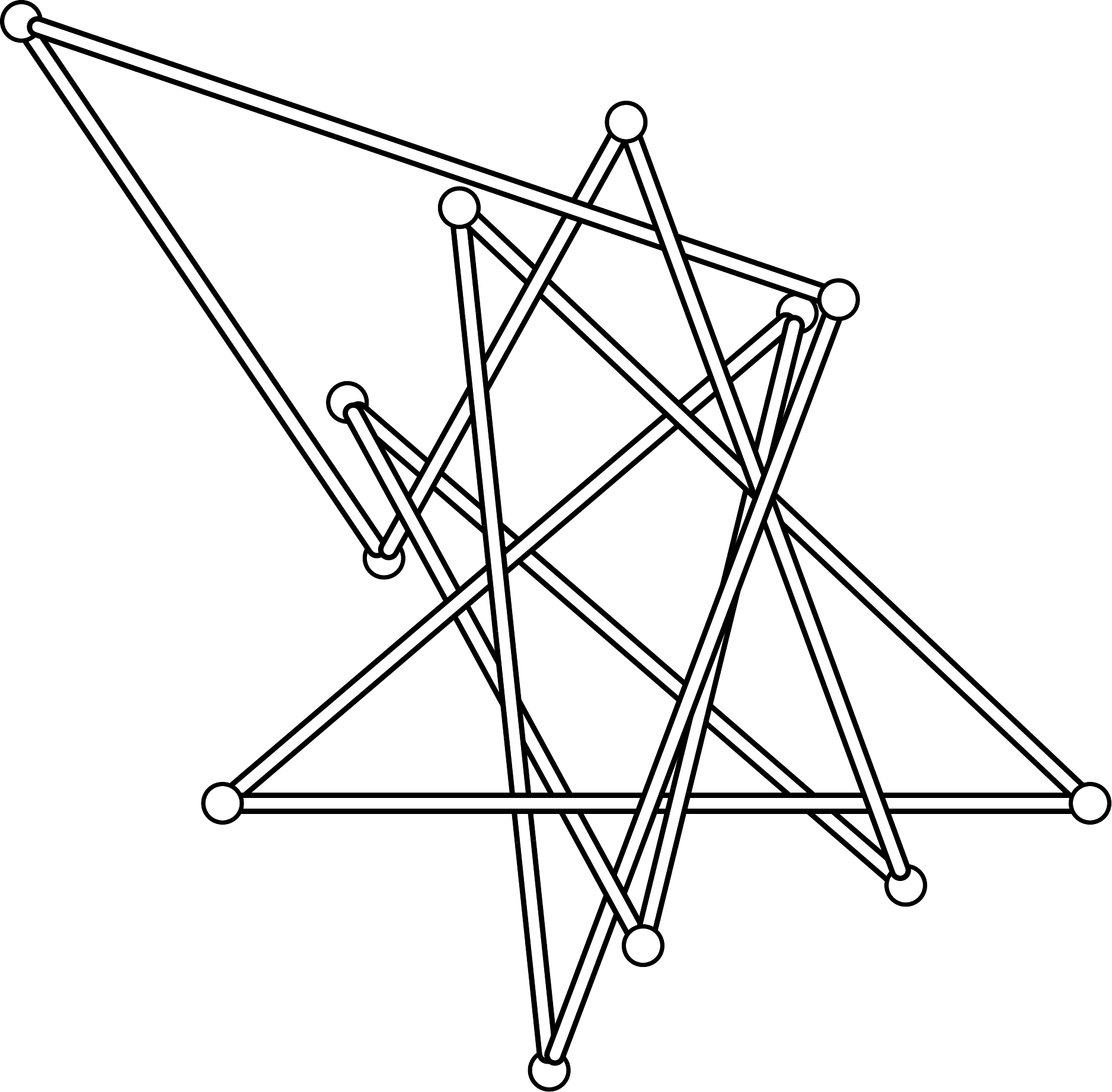}}
	& $0$ & $0$ & $0$ \\
	& $1000$ & $0$ & $0$ \\
	& $273$ & $687$ & $0$ \\
	& $377$ & $-308$ & $-27$ \\
	& $710$ & $581$ & $289$ \\
	& $-232$ & $901$ & $190$ \\
	& $186$ & $282$ & $-475$ \\
	& $465$ & $785$ & $343$ \\
	& $788$ & $-95$ & $-7$ \\
	& $144$ & $462$ & $-532$ \\
	& $485$ & $-164$ & $169$ \\
	& $662$ & $565$ & $-492$ \\
	\end{tabular*}
	\[(1,1,1,251677634,1,1,221757579,5,29397800,2012040,102253303,35434657)\]

	\medskip

	\begin{tabular*}{0.75\textwidth}{C{2.2in} R{.4in} R{.4in} R{.4in}}
	\multicolumn{4}{c}{$12n_{66}$} \\
	\cline{1-4}\noalign{\smallskip}
	\multirow{12}{*}{\includegraphics[height=1.8in,width=1.8in,keepaspectratio]{K12n66.pdf}}
	& $0$ & $0$ & $0$ \\
	& $1000$ & $0$ & $0$ \\
	& $318$ & $731$ & $0$ \\
	& $-568$ & $607$ & $448$ \\
	& $323$ & $204$ & $238$ \\
	& $-350$ & $406$ & $-474$ \\
	& $-378$ & $262$ & $515$ \\
	& $413$ & $277$ & $-97$ \\
	& $-460$ & $-210$ & $-129$ \\
	& $-130$ & $523$ & $466$ \\
	& $420$ & $-86$ & $-106$ \\
	& $-96$ & $740$ & $124$ \\
	& $865$ & $501$ & $-13$ \\
	\end{tabular*}
	 \[\setlength\arraycolsep{2pt}\begin{tiny}\begin{bmatrix}$1$ & $1$ & $1$ & $1$ & $1$ & $1$ & $1$ \
	& $1$ & $1$ & $1$ & $1$ & $1$ & $1$ \\
	$1$ & $1$ & $1$ & $1$ & $1$ & $1$ & $1$ & $1$ & $1$ & $1$ & $1$ & $1$ \
	& $1059699457$ \\
	$1$ & $1$ & $1$ & $1$ & $1$ & $1$ & $1$ & $1$ & $1$ & $1$ & $1$ & \
	$2185397635$ & $1$ \\
	$1$ & $1$ & $70743408$ & $11295467$ & $118856988$ & $1$ & $1$ & $1$ & \
	$1$ & $1$ & $1$ & $824304355$ & $1$ \\
	$1367392$ & $26039894$ & $950806273$ & $1$ & $1$ & $1314590355$ & $1$ \
	& $5140073610$ & $1$ & $1$ & $1$ & $1$ & $1$ \\
	$1$ & $1$ & $4$ & $1$ & $1$ & $1$ & $1581888017$ & $1$ & $1$ & \
	$2526630451$ & $1$ & $1996669148$ & $232048990$ \\
	$1$ & $2$ & $1$ & $1$ & $1$ & $1664395333$ & $1$ & $2917208982$ & \
	$6703631479$ & $1$ & $1773528185$ & $1$ & $1$ \\
	$2$ & $1$ & $1$ & $18$ & $213805076$ & $1$ & $1453122638$ & \
	$7367098343$ & $1$ & $574251305$ & $1$ & $1$ & $1$ \\
	$22$ & $67668619$ & $10$ & $12092360$ & $1$ & $1$ & $1$ & $1$ & $1$ & \
	$1$ & $1$ & $1$ & $2$ \\
	$543662$ & $755483$ & $946422359$ & $6992$ & $45253597$ & $358554124$ \
	& $38580721$ & $707787260$ & $14859163$ & $105984642$ & $155554867$ & \
	$554679128$ & $8970836$ \\
	$8453350$ & $20702372$ & $56711902$ & $98579$ & $127573372$ & \
	$212625391$ & $1042794420$ & $2338921294$ & $3018257581$ & \
	$1941994036$ & $1435576911$ & $1438067031$ & $725385343$ \\
	$5322026$ & $10050811$ & $13967658$ & $1505123$ & $72512945$ & \
	$2319156951$ & $517848979$ & $678503703$ & $4892804723$ & $169858433$ \
	& $2031348577$ & $444888144$ & $426467670$ \\
	$12969476$ & $63055476$ & $1262659$ & $741133$ & $1571326$ & \
	$1789635077$ & $2070400100$ & $39285759$ & $2181390421$ & \
	$1436032321$ & $3297066674$ & $1478778192$ & \
	$69197486$\end{bmatrix}\end{tiny}\]

	\medskip

	\begin{tabular*}{0.75\textwidth}{C{2.2in} R{.4in} R{.4in} R{.4in}}
	\multicolumn{4}{c}{$12n_{219}$} \\
	\cline{1-4}\noalign{\smallskip}
	\multirow{12}{*}{\includegraphics[height=1.8in,width=1.8in,keepaspectratio]{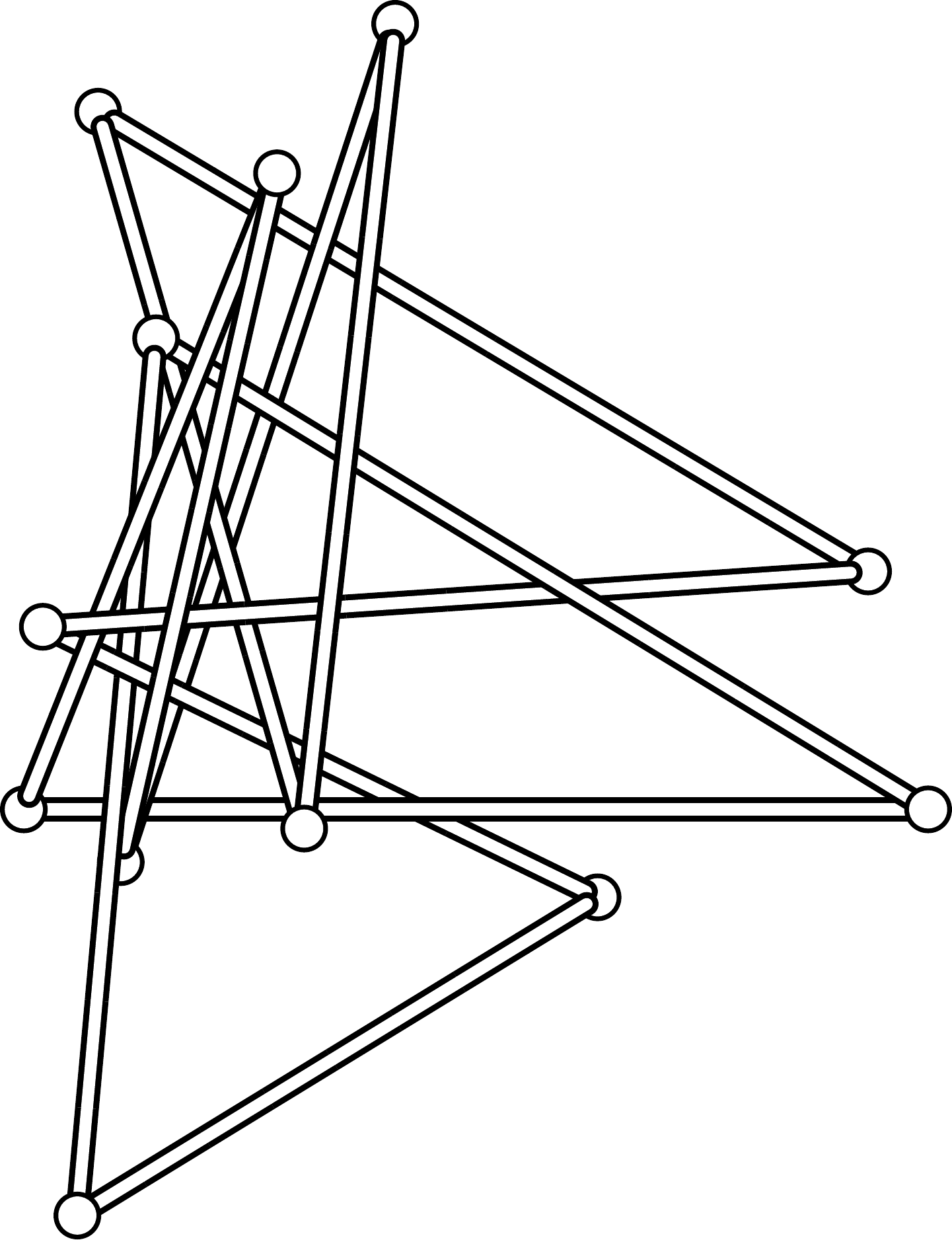}}
	& $0$ & $0$ & $0$ \\
	& $1000$ & $0$ & $0$ \\
	& $146$ & $521$ & $0$ \\
	& $59$ & $-449$ & $229$ \\
	& $635$ & $-97$ & $-510$ \\
	& $21$ & $202$ & $221$ \\
	& $934$ & $263$ & $-183$ \\
	& $82$ & $771$ & $-314$ \\
	& $310$ & $-21$ & $252$ \\
	& $411$ & $868$ & $-194$ \\
	& $107$ & $-58$ & $29$ \\
	& $280$ & $703$ & $654$ \\
	\end{tabular*}
	\[(1,1,4642857341,1,1,1,369809977,1,3334333618,657101686,872056242,112908399)\]

	\medskip

	\begin{tabular*}{0.75\textwidth}{C{2.2in} R{.4in} R{.4in} R{.4in}}
	\multicolumn{4}{c}{$12n_{225}$} \\
	\cline{1-4}\noalign{\smallskip}
	\multirow{12}{*}{\includegraphics[height=1.8in,width=1.8in,keepaspectratio]{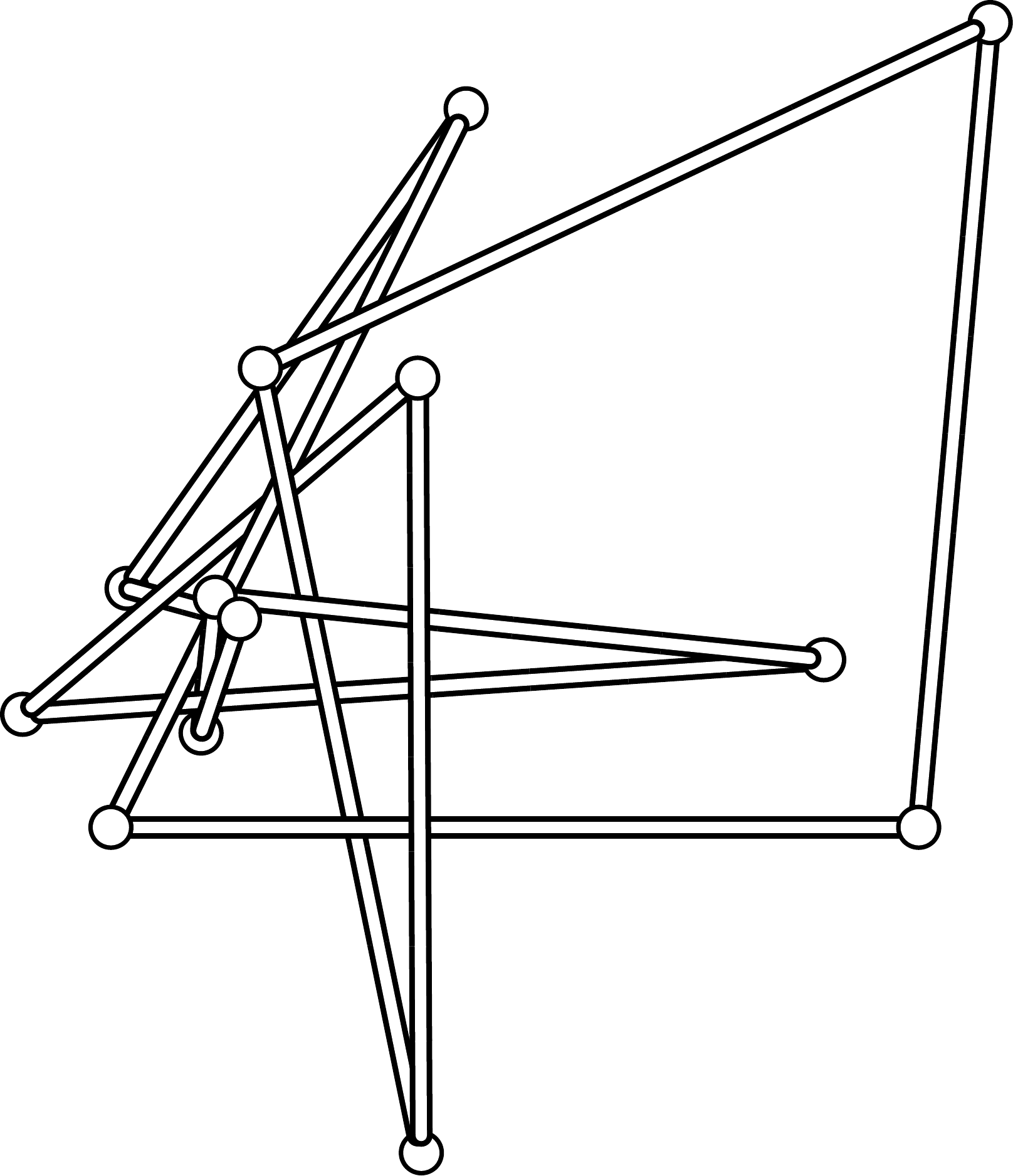}}
	& $0$ & $0$ & $0$ \\
	& $1000$ & $0$ & $0$ \\
	& $1089$ & $996$ & $0$ \\
	& $185$ & $568$ & $26$ \\
	& $384$ & $-403$ & $-105$ \\
	& $380$ & $556$ & $179$ \\
	& $-109$ & $140$ & $-588$ \\
	& $882$ & $207$ & $-479$ \\
	& $129$ & $285$ & $175$ \\
	& $112$ & $117$ & $-811$ \\
	& $159$ & $258$ & $178$ \\
	& $20$ & $296$ & $-812$ \\
	& $440$ & $889$ & $-125$ \\
	\end{tabular*}
	 \[\setlength\arraycolsep{2pt}\begin{tiny}\begin{bmatrix}$1$ & $1$ & $1$ & $1$ & $1$ & $1$ & $1$ \
	& $1$ & $1$ & $1$ & $1$ & $1$ & $1$ \\
	$1$ & $1$ & $1$ & $1$ & $1$ & $1$ & $1$ & $1$ & $1$ & $1$ & $1$ & $1$ \
	& $1$ \\
	$3751762692$ & $1$ & $1$ & $1$ & $1$ & $1$ & $1$ & $1$ & $1$ & $1$ & \
	$1$ & $163228078$ & $148375996$ \\
	$1$ & $7160242151$ & $260235613$ & $1$ & $1$ & $1$ & $1$ & $1$ & $1$ \
	& $1$ & $1$ & $198569226$ & $1$ \\
	$1$ & $1$ & $1$ & $38177126$ & $1$ & $1$ & $1$ & $6847817933$ & $1$ & \
	$1$ & $3581695$ & $1$ & $1$ \\
	$1$ & $1$ & $36573690$ & $1$ & $1$ & $1$ & $1$ & $1$ & $1$ & \
	$111133340$ & $11587746$ & $1$ & $683454611$ \\
	$1$ & $1$ & $1$ & $1$ & $1$ & $1$ & $29787011635$ & $5080355336$ & \
	$1262278581$ & $45119884$ & $1$ & $1$ & $1$ \\
	$1$ & $3580121079$ & $1$ & $1$ & $1$ & $3940844432$ & $1$ & \
	$3374642279$ & $512482248$ & $1$ & $1$ & $1$ & $1$ \\
	$3$ & $1$ & $1$ & $329721838$ & $439165299$ & $1$ & $5$ & $1$ & $1$ & \
	$5$ & $8$ & $1$ & $1$ \\
	$167594410$ & $2869603788$ & $2128863$ & $64630156$ & $71964989$ & \
	$3731403869$ & $49275204$ & $56450326$ & $336533$ & $94941175$ & \
	$84830$ & $41137220$ & $2068298$ \\
	$1088664213$ & $162014574$ & $8093340$ & $191907757$ & $370465301$ & \
	$1673424440$ & $260733778$ & $802736034$ & $993049610$ & $20483053$ & \
	$43256189$ & $302803346$ & $162835796$ \\
	$3378798968$ & $2000442164$ & $113361702$ & $217292196$ & $497047907$ \
	& $178470956$ & $4479118220$ & $228075334$ & $343333001$ & $6910835$ \
	& $4848523$ & $9094738$ & $514929315$ \\
	$2713806913$ & $630680383$ & $285764261$ & $15855082$ & $263187947$ & \
	$478384430$ & $32256762603$ & $170712380$ & $16313219$ & $1073038$ & \
	$46353686$ & $500545558$ & $88464298$\end{bmatrix}\end{tiny}\]

	\medskip

	\begin{tabular*}{0.75\textwidth}{C{2.2in} R{.4in} R{.4in} R{.4in}}
	\multicolumn{4}{c}{$12n_{553}$} \\
	\cline{1-4}\noalign{\smallskip}
	\multirow{12}{*}{\includegraphics[height=1.8in,width=1.8in,keepaspectratio]{K12n553.pdf}}
	& $0$ & $0$ & $0$ \\
	& $1000$ & $0$ & $0$ \\
	& $46$ & $301$ & $0$ \\
	& $704$ & $-423$ & $-209$ \\
	& $603$ & $-122$ & $740$ \\
	& $84$ & $-378$ & $-77$ \\
	& $869$ & $84$ & $337$ \\
	& $173$ & $-619$ & $480$ \\
	& $186$ & $190$ & $-107$ \\
	& $696$ & $-273$ & $618$ \\
	& $906$ & $417$ & $-74$ \\
	\end{tabular*}

\end{center}

\bibliography{stickknots-special,stickknots}

\end{document}